\documentclass[11pt]{article}

\usepackage[utf8]{inputenc}
\usepackage[dvips]{graphicx}
\usepackage{calc}
\usepackage{color}
\usepackage{amsmath}
\usepackage{amssymb}
\usepackage{amscd}
\usepackage{amsthm}
\usepackage{amsbsy}
\usepackage{delarray}
\usepackage{enumerate}
\usepackage[T1]{fontenc}
\usepackage{inputenc}
\usepackage{enumerate}
\usepackage[all]{xy}
\usepackage{hyperref}
\usepackage{mathtools}
\usepackage{float}
\usepackage{pdfsync}
\usepackage{textcomp}
\usepackage{transparent}

\begin{document}

\setlength{\parindent}{5mm}
\renewcommand{\leq}{\leqslant}
\renewcommand{\geq}{\geqslant}
\newcommand{\N}{\mathbb{N}}
\newcommand{\bbB}{\mathbb{B}}
\newcommand{\bbC}{\mathbb{C}}
\newcommand{\bbP}{\mathbb{P}}
\newcommand{\bbQ}{\mathbb{Q}}
\newcommand{\bbR}{\mathbb{R}}
\newcommand{\bbS}{\mathbb{S}}
\newcommand{\bbT}{\mathbb{T}}
\newcommand{\bbZ}{\mathbb{Z}}
\newcommand{\F}{\mathbb{F}}
\newcommand{\g}{\mathfrak{g}}
\newcommand{\h}{\mathfrak{h}}
\newcommand{\RN}{\mathbb{R}^{2n}}
\newcommand{\ci}{c^{\infty}}
\newcommand{\derive}[2]{\frac{\partial{#1}}{\partial{#2}}}
\newcommand{\eps}{\varepsilon}
\newcommand{\B}{\mathbf{B}}
\newcommand{\A}{\mathcal{A}}
\newcommand{\cB}{\mathcal{B}}
\newcommand{\cE}{\mathcal{E}}
\newcommand{\cF}{\mathcal{F}}
\newcommand{\cP}{\mathcal{P}}
\newcommand{\cQ}{\mathcal{Q}}
\newcommand{\cT}{\mathcal{T}}
\newcommand{\ind}{\mathrm{L}}
\newcommand{\Conj}{\mathrm{Conj}}
\newcommand {\sign}{\mathrm {sign}}

\newcommand{\BX}{\mathbf{B}(X)}

\newcommand{\pfrac}[2]{\frac{\partial #1}{\partial #2}}
\newcommand{\Image}{{\rm Im}}
\newcommand{\Ker}{\rm Ker}
\newcommand{\Hom}{{\rm Hom}}
\newcommand{\hHom}{{\mathcal Hom}}
\newcommand{\Coker}{{\rm Coker}}
\newcommand{\Coim}{{\rm Coim}}
\newcommand{\codim}{{\rm codim}}

\newcommand{\Ob}{{\rm Ob}}
\newcommand{\Mor}{{\rm Mor}}
\newcommand{\Fun}{{\rm Fun}}
\newcommand{\Nat}{{\rm  Nat}}
\newcommand{\Hh}{{\mathcal H}}
\newcommand{\I}{{\mathcal I}}
\newcommand{\cI}{{\mathcal I}^{\bullet}}
\newcommand{\cH}{{\mathcal H}^{\bullet}}
\newcommand{\cK}{{\mathcal K}^{\bullet}}
\newcommand{\cU}{{\mathcal U}}
\newcommand{\T}{{\mathcal T}}
\newcommand{\J}{{\mathcal J}}
\newcommand{\M}{{\mathcal M}}
\renewcommand{\L}{{\mathcal L}}
\renewcommand{\O}{{\mathcal O}}
\newcommand \id {{\rm id}}
\newcommand \Cat [1] {{\rm{\bf #1}}}
\newcommand{\supp}{{\rm supp}}

\DeclarePairedDelimiter{\ceil}{\lceil}{\rceil}

\theoremstyle{plain}
\newtheorem{theo}{Theorem}
\newtheorem*{theo*}{Theorem}
\newtheorem{prop}[theo]{Proposition}
\newtheorem{lemma}[theo]{Lemma}
\newtheorem{definition}[theo]{Definition}
\newtheorem*{notation*}{Notation}
\newtheorem*{notations*}{Notations}
\newtheorem{corol}[theo]{Corollary}
\newtheorem{conj}[theo]{Conjecture}
 \newtheorem{question}{Question}
\newtheorem*{conj*}{Conjecture}
\newtheorem*{claim*}{Claim}
\newtheorem{claim}[theo]{Claim}

\newcounter{numexo}[section]
\newtheorem{exo}[theo]{Exercice}
\newtheorem{exos}[numexo]{Exercices}

\newenvironment{demo}[1][]{\addvspace{8mm} \emph{Proof #1.
    ~~}}{~~~$\Box$\bigskip}

\newlength{\espaceavantspecialthm}
\newlength{\espaceapresspecialthm}
\setlength{\espaceavantspecialthm}{\topsep} \setlength{\espaceapresspecialthm}{\topsep}

\newtheorem{exple}[theo]{Example}
\renewcommand{\theexple}{}
\newenvironment{example}{\begin{exple}\rm }{\hfill $\blacktriangleleft$\end{exple}}

\newenvironment{remark}[1][]{\refstepcounter{theo} 
\vskip \espaceavantspecialthm \noindent \textsc{Remark~\thetheo
#1.} }%
{\vskip \espaceapresspecialthm}

\newenvironment{remarks}[1][]{\refstepcounter{theo} 
\vskip \espaceavantspecialthm \noindent \textsc{Remarks~\thetheo
#1.} }%
{\vskip \espaceapresspecialthm}

\def\bb#1{\mathbb{#1}} \def\m#1{\mathcal{#1}}

\def\del{\partial}
\def\Int{\mathrm{Int}}
\def\co{\colon\thinspace}
\def\Homeo{\mathrm{Homeo}}
\def\Conv{\mathrm{Conv}}
\def\Hameo{\mathrm{Hameo}}
\def\Diff{\mathrm{Diff}}
\def\Symp{\mathrm{Symp}}
\def\Sympeo{\mathrm{Sympeo}}
\def\Clos{\mathrm{Clos}}
\def\Id{\mathrm{Id}}
\newcommand{\norm}[1]{||#1||}
\def\Ham{\mathrm{Ham}}
\def\Hamtilde{\widetilde{\mathrm{UHam}}}
\def\UHam{\mathrm{UHam}}
\def\Vol{\mathrm{Vol}}

\def\Crit{\mathrm{Crit}}
\def\Spec{\mathrm{Spec}}
\def\EssSpec{\mathrm{EssSpec}}
\def\dbot{d_{\mathrm{bot}}}
\def\Leb{\mathrm{Leb}}
\def\Fix{\mathrm{Fix}}

\newcommand{\bigslant}[2]{{\raisebox{.2em}{$#1$}\left/\raisebox{-.2em}{$#2$}\right.}}

\definecolor{sobhan}{rgb}{0,.6,0}\newcommand{\sobhan}{\color{sobhan}}
\definecolor{fred}{rgb}{0,0,0.8}\newcommand{\fred}{\color{fred}}


\title{The Anosov-Katok method and pseudo-rotations in symplectic dynamics}
\author{Fr\'ed\'eric Le Roux, Sobhan Seyfaddini}
\date{\today} \date{}

\maketitle
\begin{abstract} 
We prove that toric symplectic manifolds admit Hamiltonian pseudo-rotations with a finite, and in a sense minimal, number of ergodic measures.   The set of ergodic measures of these pseudo-rotations consists of the measure induced by the symplectic volume form  and the  Dirac measures supported at the fixed points of the torus action.   Our construction relies on the conjugation method of Anosov and Katok.

\flushright \emph{Dedicated to Claude Viterbo. Joyeux anniversaire, Claude !}

\end{abstract}
\tableofcontents

\section{Introduction}
The main goal of this paper is to exhibit examples of Hamiltonian diffeomorphisms on symplectic manifolds of dimension greater than two, which, on the one hand, have a finite number of periodic points, and on the other hand, have interesting and complicated dynamics.  We will refer to Hamiltonian diffeomorphisms with a finite number of periodic points as Hamiltonian {\bf pseudo-rotations}.\footnote{There exist several working definitions of Hamiltonian pseudo-rotations in the literature; see  \cite[Def.\ 1.1]{CGG19b} and the discussion therein.  The examples we construct in this article do satisfy the requirements in all definitions known to us.}  Such diffeomorphisms have been of great interest in  dynamical systems and symplectic topology; see, for example, \cite{Anosov-Katok, Fathi-Herman, Fayad-Katok, BCL04, BCL06, Bramham15a, Bramham15b, LeCalvez16, AFLXZ, Ginzburg-Gurel18a, Ginzburg-Gurel18a, CGG19a, CGG19b, Shelukhin19a, Shelukhin19b}. 


The only closed surface admitting Hamiltonian pseudo-rotations is the sphere with the simplest examples being irrational rotations.\footnote{Closed surfaces of positive genus fall under the class of symplectic manifolds which satisfy the Conley conjecture \cite{Hingston, Ginzburg, LeCalvez06} asserting that any Hamiltonian diffeomorphism has infinitely many simple periodic points.  Clearly, such symplectic manifolds do not admit pseudo-rotations.} More interesting examples, with only three ergodic\footnote{Recall that a Borel probability measure $\mu$ is said to be ergodic for  $f: M \to M$ if  it is preserved by $f$ and, furthermore, satisfies the following condition: for any Borel subset $A \subset M$ such that $f^{-1}(A) \subset A$ either $\mu(A) = 0$ or $\mu(A) =1$.} measures, were constructed by Fayad and Katok \cite{Fayad-Katok}, using the so-called conjugation method of Anosov-Katok \cite{Anosov-Katok}.  Note that three is the minimal possible number of ergodic measures for a Hamiltonian diffeomorphism of the sphere because any such diffeomorphism has at least two fixed points and so preserves, in addition to the area, the Dirac delta measures supported at the fixed points.  

Consider a closed symplectic manifold $(M, \omega)$ carrying a Hamiltonian circle action $S$ such that the fixed point set of the action  $\Fix(S)$ is finite and the action is locally free on $M\setminus \Fix(S)$.  Any such manifold, like the sphere and the projective spaces $\bbC P^n$, admits pseudo-rotations obtained from the irrational elements of the circle.  However, these example have very simple dynamics. As we explain below in Section \ref{sec:transitivity},  in this context, a straight-forward adaptation of the Anosov-Katok method yields pseudo-rotations which are transitive, i.e.\ have dense orbits.  The presence of a circle action as above is an indispensable ingredient of the Anosov-Katok method and, as far as we know, there are no known examples of pseudo-rotations on manifolds not carrying circle actions.

Producing pseudo-rotations with  dynamics more complicated than transitivity is more involved and is the main goal of this article.   Our main result, whose proof relies on the Anosov-Katok method, guarantees the existence of pseudo-rotations with a finite, and in a sense optimal, number of ergodic measures on {\bf toric} symplectic manifolds.  Recall that a $2n$-dimensional symplectic manifold is called toric if it admits an effective Hamiltonian action of the torus $\bb T^n$; we review this definition and relevant facts in Section \ref{sec:prelim_symp}.

\begin{theo}\label{theo.main}
Let $(M, \omega)$ be a closed toric symplectic manifold, and denote by $\ell$  the number of fixed points of the corresponding torus action. Then, $(M, \omega)$ admits a Hamiltonian pseudo-rotation $f$ with exactly $\ell+1$ ergodic measures.  The set of ergodic measures of $f$ consists of the measure induced by the symplectic volume form $\omega^n$ and the $\ell$ Dirac measures supported at the fixed points of the torus action.
\end{theo}

Applying the above theorem to $(\mathbb CP^n, \omega_{FS})$, which admits a toric action with $n$ fixed points, yields pseudo-rotations with $n+1$ ergodic measures.   

The number $\ell +1$ appears to be optimal because  the rank of the singular homology of $M$ is $\ell$; see, for example,  \cite[Theorem 3.3.1]{Cannas}.  Hence, by the Arnold conjecture any non-degenerate Hamiltonian diffeomorphism of $(M, \omega)$ has at least $\ell$ fixed points and thus must have at least $\ell + 1$ ergodic measures.  Non-degeneracy of the pseudo-rotations we produce can be guaranteed by appropriately applying the Anosov-Katok method, see Remark \ref{rem:nondeg}; moreover,  some authors incorporate non-degeneracy into the definition of pseudo-rotations, see \cite[Def.\ 1.1]{CGG19b}.

Denote by $B_r$  the standard closed Euclidean ball of radius $r$, and let  $\{B_i\}, \{B_i'\}, i = 1, \ldots, k'$, be two collections of pair-wise disjoint subsets of  $(M, \omega)$  all of which are images of $B_r$ under symplectic embeddings.   In the case of the  2-sphere, a key component of the argument is the fact that one can find a  symplectomorphism $\psi$ such that $\psi(B_i) = B_i'$ for $i= 1, \ldots, k$.  However, the existence of such $\psi$ cannot be guaranteed on higher dimensional symplectic manifolds even if $k=1$.  Indeed, under certain assumptions, there exist obstructions to the existence of such $\psi$ which are generally referred to as symplectic camel obstructions; see \cite{McDuff-Salamon} for further details.  The camel-type obstructions do disappear for sufficiently small values of $r$. Hence, by Katok's Basic Lemma \cite{Katok}, one can try to surmount these obstacles by breaking $B_i, B_i'$ into balls of sufficiently small radius to obtain a symplectomorphism $\psi$ such that $\psi(B_i) \Delta B_i'$ is of nearly, but not exactly, zero measure; here $\Delta$ stands for the symmetric difference of sets.  It appears to us that this approach could potentially yield ergodic pseudo-rotations, but we could not utilize it to construct pseudo-rotations with a finite number of ergodic measures.\footnote{As explained in~\cite{Fayad-Katok}, to get ergodicity one needs to control \emph{almost all} orbits; but in order to prove Theorem~\ref{theo.main} one needs to control \emph{all} orbits.} 
Instead, we overcome the camel-type difficulties by taking advantage of the existence of a ``nearly global'' system of action-angle coordinates on toric symplectic manifolds; see Section \ref{sec:prelim_symp}.  It is not clear to us if the assumption of $(M, \omega)$ being toric is necessary for the existence of pseudo-rotations with few ergodic measures.

Finally, we should mention that several authors have used the Anosov-Katok method to produce interesting examples, other than Hamiltonian pseudo-rotations, in higher dimensionsal symplectic manifolds.  In \cite{Katok}, Katok constructs an autonomous Hamiltonian with numerous properties including ergodicity of the restriction of the Hamiltonian flow to   its energy levels.  In \cite{Polterovich99}, Polterovich uses Katok's lemma to prove, among other results, that every closed symplectic manifold admits contractible Hamiltonian loops which are strictly ergodic; see Theorem 1.2.A therein.  In \cite{Hernandez-Presas}, Hern\`andez-Corbato and Presas, construct examples of (non-Hamiltonian) minimal symplectomorphisms and strictly ergodic contactomorphisms.

\subsubsection*{Organization of the paper}
In Sections \ref{sec:general_scheme} \& \ref{sec:fast_convergence} we present the general scheme of the Anosov-Katok method.  In Section \ref{sec:transitivity}, which may be viewed as a warm-up  for the proof of the main result, we explain how to construct transitive pseudo-rotations.  In Section \ref{sec:proof_main_thm}, we state  Proposition \ref{prop.main}, which is the key technical proposition of the paper, and we use it to prove Theorem \ref{theo.main}.  

The rest of the paper is dedicated to the proof of Proposition \ref{prop.main}.  In Section \ref{sec:prelim}, we review the relevant aspects of toric symplectic geometry and prove  preliminary symplectic lemmas which will be used in the following section.  Section \ref{sec:proof_main_prop} is the technical heart of the paper and contains the proof of Proposition \ref{prop.main}.

\subsubsection*{Acknowledgments} 
This project began as an attempt at answering a question of  Viktor Ginzburg and  Ba\c{s}ak  G\"{u}rel; we are grateful to them for stimulating discussions and their interest.  We also thank Laurent Charles, Sylvain Crovisier, Yael Karshon, Leonid Polterovich, Shu Shen, and Maxime Zavidovique for helpful comments.

SS:  This project has received funding from the European Research Council (ERC) under the European Union’s Horizon 2020 research and innovation program (grant agreement No. 851701).

\begin{figure}[H]
\centerline{\includegraphics[width=8cm]{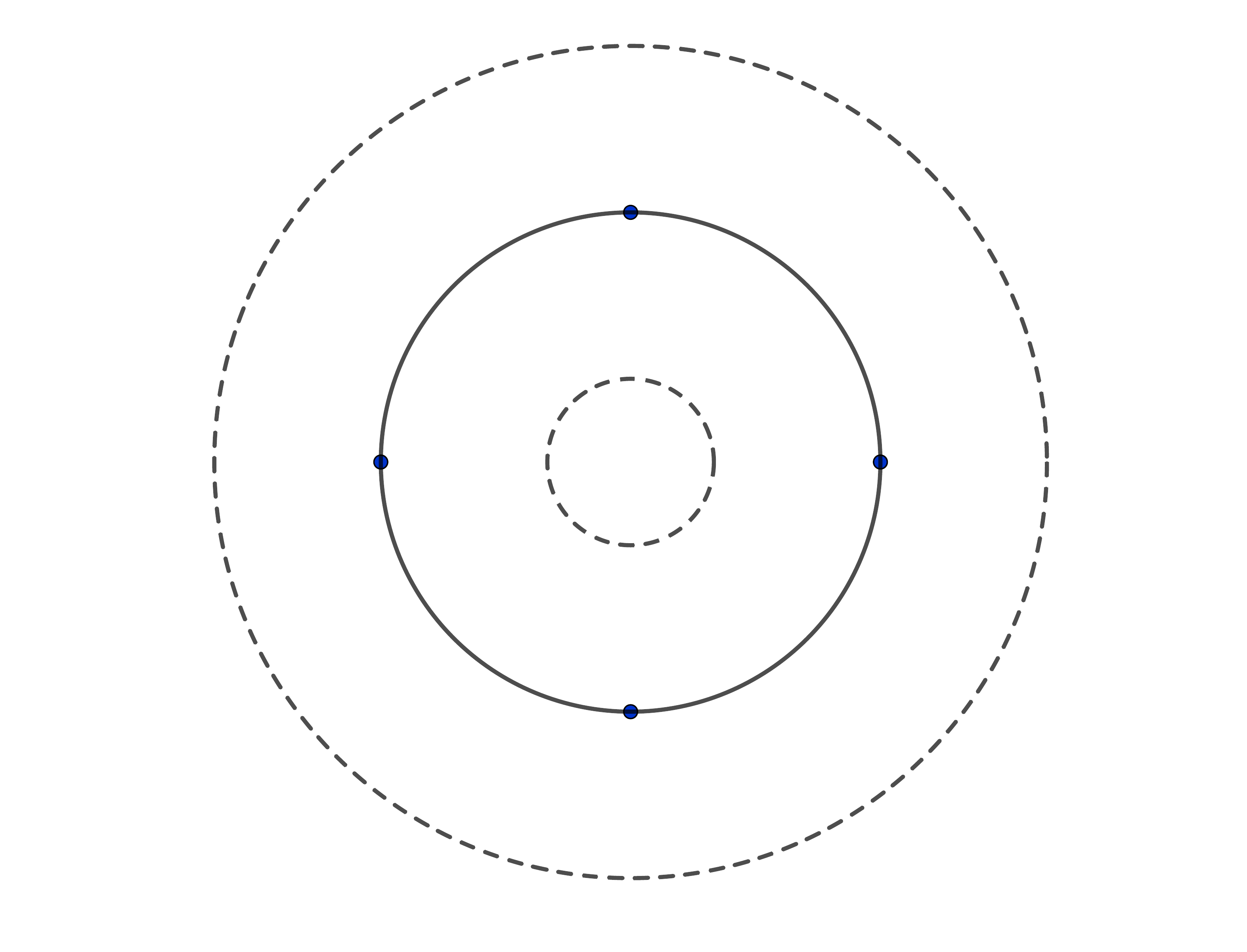}\includegraphics[width=8.5cm]{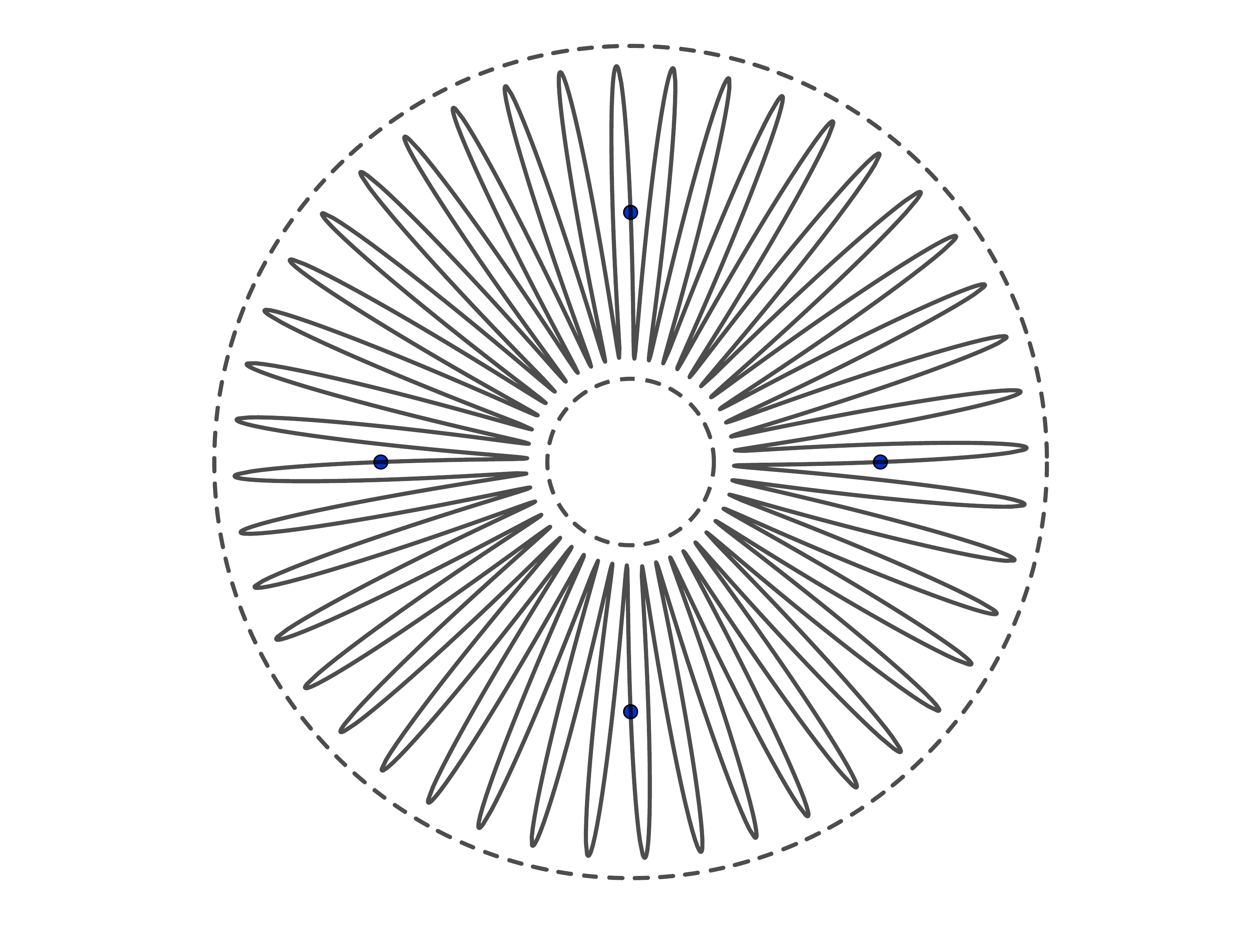}}

\centerline{\includegraphics[width=8.2cm]{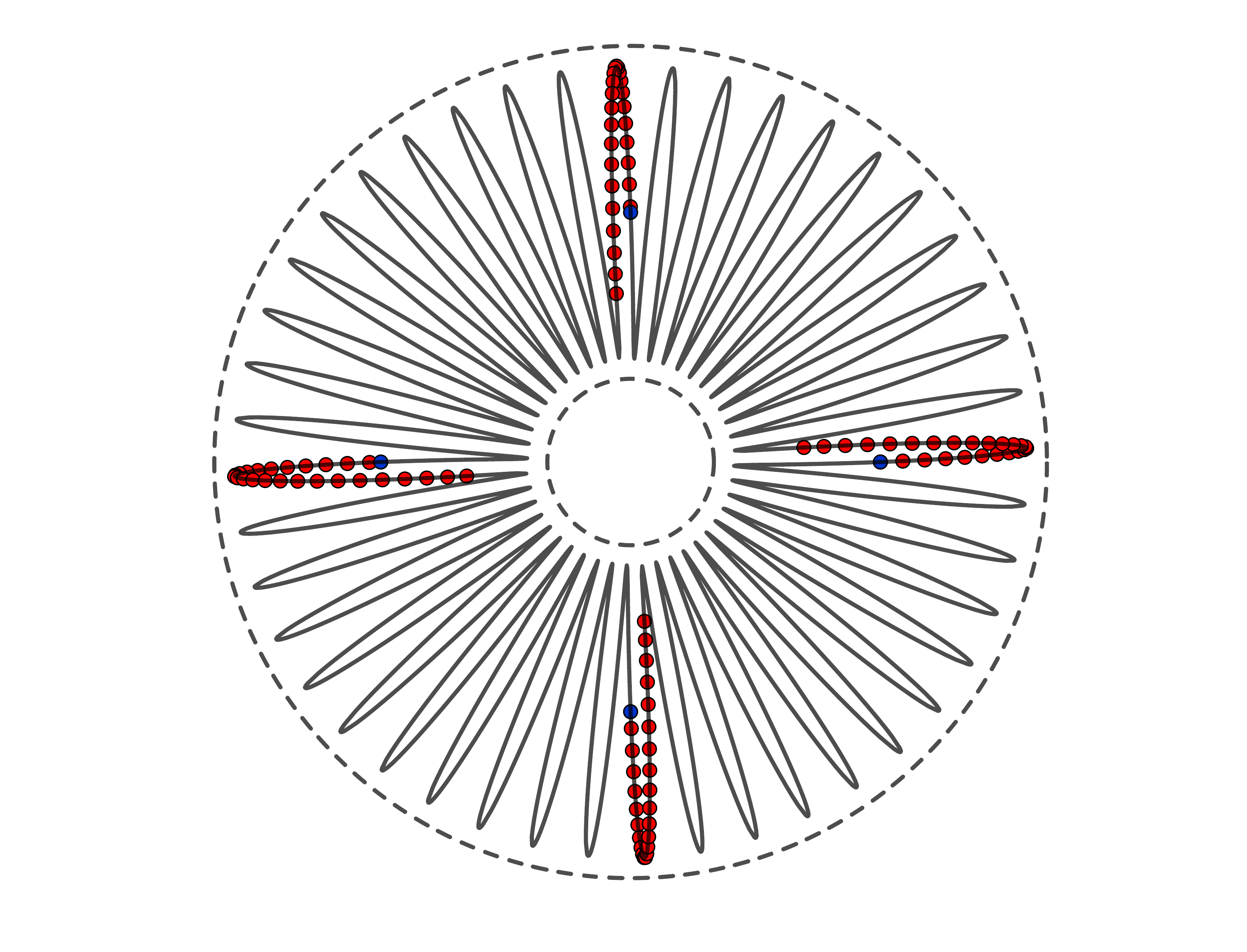}\includegraphics[width=8.2cm]{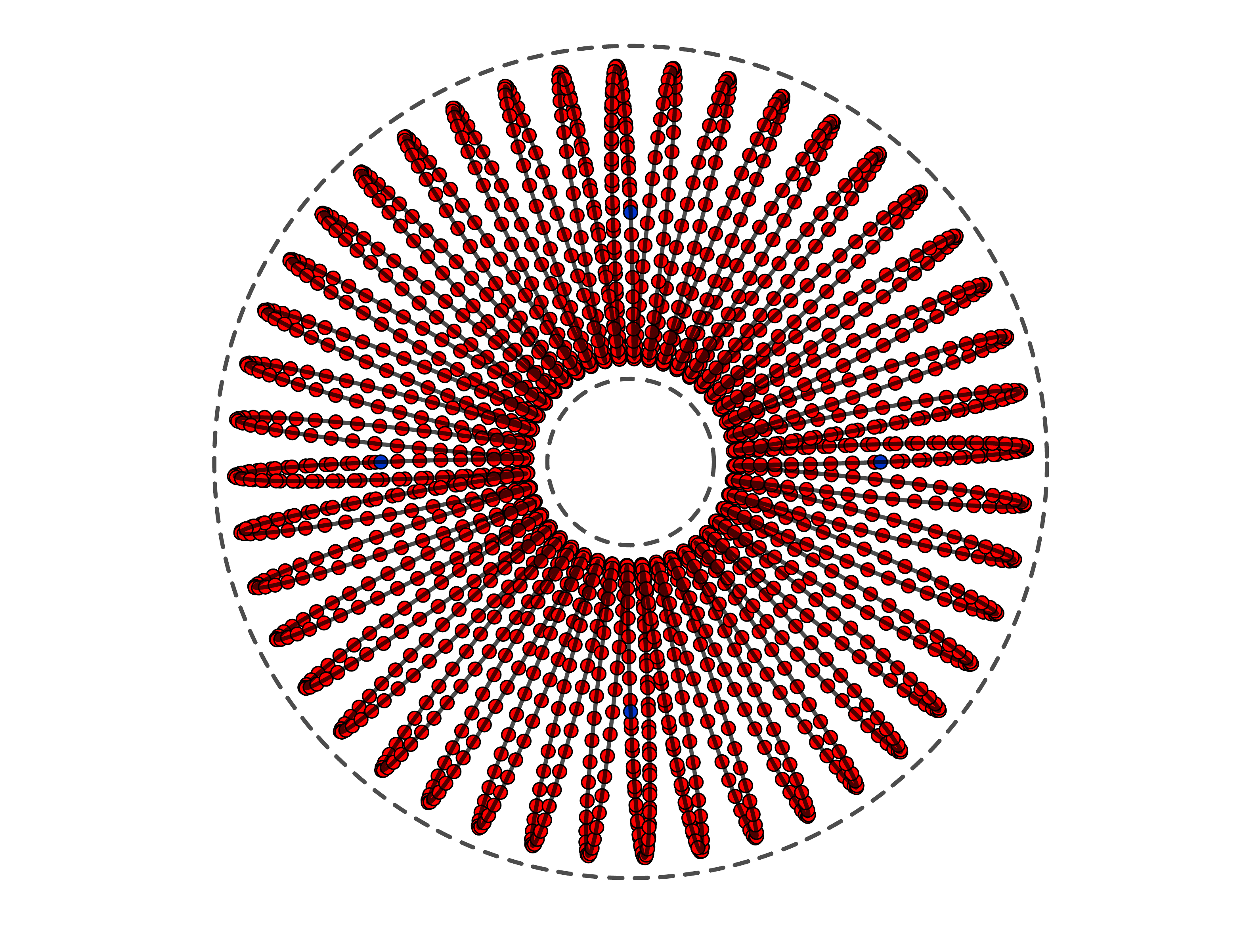}}

\begin{enumerate}
\item One orbit of the circle action $(S_{\alpha})_{\alpha \in \bbS^1}$ with one discrete orbit for $S_{\alpha_{1}}$;

\item The same orbit for $(h_{1} S_{\alpha} h_{1}^{-1})_{\alpha \in \bbS^1}$, the discrete orbit is unchanged;

\item First 100 iterates for $h_{1} S_{\alpha_{2}} h_{1}^{-1}$, with $\alpha_{2}$ very close to $\alpha_{1}$;

\item The discrete orbit is $\varepsilon_{1}$-dense for some small $\varepsilon_{1}$.
\end{enumerate}

\caption{The first two steps in Anosov-Katok method to construct a transitive diffeomorphism (here $\alpha_{1} = \frac{1}{4}$), see also Section~\ref{sec:transitivity} \label{fig.general-scheme}}
\end{figure}

\section{The conjugation method of Anosov and Katok}
Let $(M, \omega)$ be a closed symplectic manifold which admits a smooth Hamiltonian action of the circle, denoted by $(S_{\alpha})_{\alpha \in \bbS^1}$.    We identify the circle $\bbS^1$ with $\bb R / \bbZ$.  We denote the fixed point set of the action by $\Fix(S)$  and we suppose that the action is locally free outside of $\Fix(S)$.  This means that there exists a neighborhood $V$ of $0 \in \bb S^1$ such that for every non-zero $\alpha \in V$ the homeomorphism $S_{\alpha}$ has no fixed points in $M \setminus \Fix(S)$.

\subsection{General scheme}\label{sec:general_scheme}
  We outline here the general scheme of constructing a pseudo-rotation, say $f$, via the conjugation method of Anosov-Katok (see also Figure~\ref{fig.general-scheme}).   The pseudo-rotation $f$ will be obtained as the $C^{\infty}$ limit of a sequence of Hamiltonian diffeomorphisms $(f_{n})$ which are of the form
   
$$f_{n} = H_{n} S_{\alpha_{n+1}} H_{n}^{-1}$$ where $H_{n}$ is a symplectic diffeomorphism of $M$ and $\alpha_{n}=\frac{p_{n}}{q_{n}} \in \bbQ/\bbZ$ (all fractions are implicitly supposed to be irreducible).   We can start with any choice for $H_0$ and $\alpha_1$, e.g.\ $ H_0 = \Id$ and $\alpha_1 =\frac{1}{4}$.   The Hamiltonian diffeomorphism $f_n$ is obtained from $f_{n-1}$ in the following manner:  We construct a symplectic diffeomorphism $h_{n}$ 
which coincides with the identity near $\Fix(S)$ and commutes with $S_{\alpha_{n}}$; let $H_{n} := H_{n-1} \circ h_{n}$.  Observe that 
$$
H_{n} S_{\alpha_{n}} H_{n}^{-1} =  H_{n-1} S_{\alpha_{n}} H_{n-1}^{-1} = f_{n-1}.
$$
Hence, by choosing  $\alpha_{n+1}$ to be sufficiently close to $\alpha_{n}$, we can ensure that $f_{n+1}$, defined by the above formula, is as close as desired to $f_{n}$ (in the $C^{\infty}$ topology). On the one hand, this ensures the $C^{\infty}$ convergence of the sequence $(f_n)$ and on the other hand it will allow us to prove  that the limit map $f$ inherits the approximate dynamical  and ergodic properties of the $f_{n}$'s.  Note that if the convergence is fast enough, the sequence $(\alpha_{n})$ will converge to some irrational number $\alpha$, and the fixed point set of $f$ will coincide with the fixed point set of the initial circle action (see Section~\ref{sec:fast_convergence}).
  We should point out that the map $f$, being a $C^{\infty}$ limit of Hamiltonian diffeomorphisms, is indeed a Hamiltonian diffeomorphism of $M$; this non-trivial fact is a consequence  of the $C^{\infty}$ Flux conjecture which was settled by Ono \cite{Ono_Flux}.\footnote{It can easily be seen that $f$ is isotopic to the identity in $\Symp(M, \omega)$.  Hence, $f$ is automatically Hamiltonian if  $H^1(M)$ is trivial, which is the case for all toric symplectic manifolds. }

\medskip

To ensure that the Hamiltonian diffeomorphism  $h_n$ commutes with $S_{\alpha_n}$, we will carry out the construction of $h_n$ on the quotient of $M\setminus \Fix(S)$ 
by the action of the rotation $S_{\frac{1}{q_n}}$.  To make sure that this quotient is a smooth manifold we must pick $q_n$ such that the action of  $S_{\frac{1}{q_n}}$ on $M \setminus \Fix(S)$ is free which  is not automatically guaranteed because the circle action is assumed to be only locally free on $M\setminus \Fix(S)$.  To overcome this technical difficulty, the numbers $\alpha_n$ will be picked as follows.  Consider the subgroup of $\bbS^1$ generated by the union of the stabilizers of all points in $M \setminus \Fix(S)$. Since the action is locally free on  $M \setminus \Fix(S)$, this subgroup is generated by some $\frac{1}{q_{0}}$.  We denote
\begin{equation}\label{eq:def_cQ}
\cQ = \{q \in \N: q \text { is relatively prime to } q_0 \}.
\end{equation}
This set is closed under multiplication. The set $\cQ$ is significant to our construction because of the following property: for every rational number $\frac{1}{q}$ with $q$ in $\cQ$, the action of $\bb Z / q \bb Z$ generated by $S_{\frac{1}{q}}$ on $M \setminus \Fix(S)$ is free.  Hence, the quotient 
$$\frac{M \setminus \Fix(S)}{ S_{\frac{1}{q}}}$$  is a smooth manifold.  Moreover, it naturally inherits the symplectic structure of $M$.  Lastly, note that the set of rationals $\frac{p}{q}$ with $q \in \cQ$ is dense in $\bb S^1$.

\begin{remark}\label{rem:nondeg}
 We do not know if pseudo-rotations are in general necessarily  non-degenerate.  However, as we now explain, it is possible to ensure that the pseudo-rotation $f$ from Theorem \ref{theo.main} is non-degenerate.  Recall that being non-degenerate means that the derivative of $f$ at any of  its fixed points does not have $1$ as an eigenvalue.
  
  The sequence of rational numbers $\alpha_n$ has a limit which we denote by $\alpha$; as noted before, fast convergence of the sequence guarantees that $\alpha$ is irrational.  The derivative of $f$ at any of its fixed points has the same eigenvalues as the derivative of $S_\alpha$ at the same fixed point.  This is because the diffeomorphisms $H_n$ coincide with the identity in a neighborhood of  $\Fix(S)$.  Hence, the diffeomorphism $f$ is non-degenerate if and only if $S_\alpha$ is non-degenerate.
  
    In the case of the circle action used in the construction of the pseudo-rotations of Theorem \ref{theo.main}, $\alpha$ being irrational guarantees that $S_\alpha$ is indeed non-degenerate.   This fact, which can be verified via the equivariant version of the Darboux theorem (see \cite[Theorem 3.1.2]{Cannas}), is a consequence of the conditions imposed by Equation \eqref{eq:locally_free_action} below.
    
    Let us add that pseudo-rotations of the $2$-sphere are necessarily non-degenerate (this is a well-known consequence of the rotation vectors theory, see~\cite{Franks1, Franks2}).  It would be interesting to know if pseudo-rotations are always non-degenerate in higher dimensions.  
\end{remark}

\subsection{Why is $f$ a pseudo-rotation?}\label{sec:fast_convergence}
   Let $f$ be a Hamiltonian diffeomorphism obtained via the conjugation method as described in the previous section.  As mentioned earlier, our goal is to construct $f$ so that it is a pseudo-rotation (i.e.\ no  periodic points outside of $\Fix(S)$) and it displays complicated dynamical behaviour (transitivity, ergodicity, unique ergodicity).  Being a pseudo-rotation is automatic for the $f$ satisfying the conclusion of Theorem~\ref{theo.main} because every periodic orbit supports an ergodic measure.  
In this section, we will explain how the conjugation method can easily provide a transitive map; this can be considered as a warm-up for the more technical proof of Theorem~\ref{theo.main}.
As we will now  explain, if the convergence of the sequence $(f_n)$ is fast enough, which can be arranged by choosing the $\alpha_n$'s appropriately, then the limit map $f$ will also be a pseudo-rotation.   The general idea, which will be useful in many places, e.g.\  to get transitivity, is that \emph{an open property which holds at some step $n$ is automatically transmitted to the limit map ``by fast convergence''}.

%


Here are the details. Prior to the construction, we fix an increasing sequence $(K_{n})$  of compact subsets of $M \setminus \Fix(S)$ whose union is $M \setminus \Fix(S)$.   At the beginning of step $n$ of the construction, we define the set 
$ U(K_{n}, q_{n}) \subset \Ham(M, \omega)$ consisting of all Hamiltonian diffeomorphisms which have no periodic points of period less than $q_n$  inside $K_{n}$.  (Recall that $q_n$ is determined in step $n-1$.)  This set contains $f_{n-1}$ and is open in the $C^0$-topology.  We choose a $C^{\infty}$-neighborhood $V_{n}$ of $f_{n-1}$ whose $C^{\infty}$-closure is contained in $U(K_n, q_{n})$.  At each subsequent step $i \geq n$ we simply choose the number $\alpha_{i}$ close enough to $\alpha_{i-1}$ so that the Hamiltonian diffeomorphism $f_i$ belongs to $V_{n} $.  This implies that the limit map $f$ belongs to $U(K_{n}, q_{n})$, that is, $f$ will have no periodic points of period less than $q_n$  inside $K_{n}$.  The map $f$ will have this property for every $n$, and the numbers $q_n$ will be picked such that $q_n \to \infty$, which ensures that it will have no periodic points in $M\setminus \Fix(S)$.


\subsection{Transitivity}\label{sec:transitivity}
In this section we explain how to construct a transitive pseudo-rotation with the previous general scheme.  This construction is illustrated by Figure~\ref{fig.general-scheme}.
Every Anosov-Katok construction in Hamiltonian dynamics requires  certain information on abundance of symplectic diffeomorphisms. For transitivity the required information is very light: we need the fact that the symplectic group acts transitively on $p$-tuples of distinct points, as expressed by Lemma \ref{lemma.transitivity}  below.

We say that a subset $Z$ of $M$ is \emph{$\varepsilon$-dense} in $M$ if the open balls of radius $\varepsilon$ around each point of $Z$ cover $M$. Note that when $\varepsilon$-density holds for $Z$ it also holds for some finite subset $Z'$ of $Z$, by compactness, and for any subset $Z''$ close enough to $Z'$.

Let $(\varepsilon_{n})$ be any fixed sequence of positive numbers that converges to $0$.
Assume inductively that $\alpha_{n}, H_{n-1}$ have been constructed as above, with 
$f_{n-1} = H_{n-1} S_{\alpha_{n}} H_{n-1}^{-1}$ which is \emph{$\varepsilon_{n-1}$-transitive}: there exists some point whose orbit is $\varepsilon_{n-1}$-dense, and more precisely
there exists some $x_{n-1} \in M$ and $N_{n-1}>0$ such that the balls
$$
B_{\varepsilon_{n-1}}(x_{n-1}), \dots , B_{\varepsilon_{n-1}}(f_{n-1}^{N_{n-1}}(x_{n-1})),
$$ 
cover $M$. We now explain how to construct $h_{n}, \alpha_{n+1}$ so that $f_{n}$ is $\varepsilon_{n}$-transitive and arbitrarily close to $f_{n-1}$ (see also Figure~\ref{fig.general-scheme}).

Let $U=M \setminus \mathrm{Fix}(S)$, and consider the quotient map
$$
\pi: U \to M' = U/S_{\alpha_{n}}.
$$
The space $M'$ is a smooth manifold, on which $\omega$ induces a symplectic structure. Symplectic diffeomorphisms of $M'$ corresponds to symplectic diffeomorphisms of $U$ that commutes with $S_{\alpha_{n}}$. The Hamiltonian circle action $S$ induces a Hamiltonian circle action on $M'$.
First choose some finite subset $F_{2}$ of $M'$ such that $\pi^{-1}(F_{2})$ is $\eta_{n}$-dense in $M$, where $\eta_{n}$ is provided by the continuity of $H_{n-1}$: 
any two points that are $\eta_{n}$-close in $M$ have their images $\varepsilon_{n}$-close.
Then choose some finite set $F_{1}$ of $M'$ included in a single orbit of the circle action, and which has the same cardinality as $F_{2}$. By the next lemma, there is a compactly supported symplectic diffeomorphism of $M'$ that sends $F_{1}$ to $F_{2}$. We lift this diffeomorphism to a symplectic diffeomorphism $h_{n}$ of $M$ which is the identity near $\mathrm{Fix}(S)$ and  sends $\pi^{-1}(F_{1})$ to $\pi^{-1}(F_{2})$.

\begin{lemma}\label{lemma.transitivity}
Given a symplectic manifold $M$, and an integer $p>0$, the group of (compactly supported) symplectic diffeomorphisms of $M$ acts transitively on $p$-tuples of distinct points:
for every $(x_{1}, \dots, x_{p}), (y_{1}, \dots, y_{p}) \in M^p$ with $x_{i} \neq x_{j}, y_{i} \neq y_{j}$ for every $i \neq j$, there is some $\Phi \in \Symp_{0}(M)$ such that $\Phi(x_{i})=y_{i}$ for every $i=1, \dots, p$.
\end{lemma}

Now denote by $C$ the orbit of the circle action $S$ which contains $\pi^{-1}(F_{1})$. Note that $h_{n}(C)$ is $\eta_{n}$-dense in $M$, and thus $H_{n-1}(h_{n}(C))$ is $\varepsilon_{n}$-dense in $M$. Choose $\alpha_{n+1}$ to be  a rational number very close to $\alpha_{n}$, so that there is some point in $C$ whose (discrete) orbit $C'$ under $S_{\alpha_{n+1}}$ is $\eta_n$-dense in $C$.  Then, $H_{n-1}(h_{n}(C'))$ is still $\varepsilon_{n}$-dense in $M$. Let $H_{n} = H_{n-1} \circ h_{n}$ and $f_{n} = H_{n} S_{\alpha_{n+1}} H_{n}^{-1} $ as in the general scheme, then $H_{n}(C')$ is an orbit of $f_{n}$ which is $\varepsilon_{n}$-dense, as wanted.
(Of course, we also make sure that $f_{n}$ satisfies the constraints ensuring that the limit map $f$ is a pseudo-rotation, by taking $\alpha_{n+1}$ even closer to $\alpha_{n}$ if needed, as described in the previous section.)

It remains to check  that, provided the convergence is fast enough, the limit map $f$ will have a dense orbit. At step $n$ the map $f_{n}$ has an $\varepsilon_{n}$ dense orbit. The set of maps having an $\varepsilon_{n}$-dense orbit is open in the $C^0$-topology. Thus this property will be transmitted to the limit map $f$ by fast convergence. It is  a classical Baire category argument that a map that has $\varepsilon$-dense orbits for arbitrarily small $\varepsilon$'s has a dense orbit. Thus, $f$ is transitive.
   
\subsection{Pseudo-rotations with minimal number of ergodic measures}\label{sec:proof_main_thm}

In this section, we prove  Theorem \ref{theo.main} relying on  Proposition \ref{prop.main}  below whose proof takes up the remainder of the paper.
Our standing assumption, while proving Theorem \ref{theo.main} and Proposition \ref{prop.main}, is that $(M, \omega)$ is a toric symplectic manifold and that  the locally free circle action $S$ is compatible with the torus action in the sense that it is obtained by composing the torus action with a group morphism from the circle to the $n$-torus.  We explain why $S$ as described here exists in Section \ref{sec:prelim_symp}; see the discussion around Equation \eqref{eq:locally_free_action}.


   Denote by $\cP(M)$ the space of Borel probability measures on $M$.  This space is endowed with the weak topology: a sequence $(\mu_n)$ in $\cP(M)$ converges to $\mu$ if and only if the sequence $(\int f d\mu_n)$ converges to $\int f d\mu$ for all continuous functions $f$.  Recall that if $\mu$ is a Borel probability measure on a topological space $N$ and $h : N \rightarrow M$ is a continuous mapping, then the pushforward of  $\mu$ by $h$ is  defined by the formula 
   $$
h_{*} \mu(E) := \mu(h^{-1}(E)).
$$ 
Let $\cE \subset \cP(M)$ be the set consisting of the volume measure and the $\ell$ Dirac measures supported at the fixed points of the circle action. We denote by $\Conv(\cE)$ the convex hull of $\cE$ in $\cP(M)$. 
Lastly, recall  \eqref{eq:def_cQ},  the definition of the set $\cQ \subset \N$ of permitted denominators for the $\alpha_{n}$s.

\begin{prop}
\label{prop.main}
Let $q \in \cQ$ be a positive integer and $\cU \subset \cP(M)$ an open neighborhood of $\Conv(\cE)$.  There exists $h \in \Symp_0(M, \omega)$ with the following properties:
\begin{enumerate}
\item $h$ coincides with the identity near $\Fix(S)$, 
\item $hS_{\frac{1}{q}} = S_{\frac{1}{q}}h$,
\item For every $x \in M$, the push-forward of the Lebesgue measure on the circle by the map $t \mapsto h S_{t} h^{-1}(x)$ belongs to $\cU$.
\end{enumerate} 
\end{prop}

We will now show that the above proposition implies the existence of a pseudo-rotation whose set of invariant ergodic measures is exactly $\cE$.

\begin{proof}[Proof of Theorem \ref{theo.main}]
Let $\mathcal{U}_n$ be a sequence of open neighborhoods of $\Conv(\cE)$, such that  
$$\bigcap \mathcal{U}_n = \Conv(\cE).$$

As explained above, we start step $n$ with the rational number  $\alpha_{n} = \frac{p_{n}}{q_{n}}$ and  the maps 
$$
H_{n-1}, \ \ f_{n-1} = H_{n-1} S_{\alpha_{n}} H_{n-1}^{-1}.
$$
Since $H_{n-1}$ is symplectic and fixes $\Fix(S)$, its action $(H_{n-1})_{*}$ on $\cP (M)$ fixes every element of $\Conv(\cE)$. Thus,  $(H_{n-1})^{-1}_{*} (\cU_{n})$ is an open neighborhood of $\Conv(\cE)$. 

We denote the Lebesgue measure on the circle by  $\mathrm{Leb}_{\bbS^1}$.
Applying Proposition~\ref{prop.main} to the integer $q_n$ and the set $(H_{n-1})_{*}^{-1} (\cU_{n})$,
 we obtain $h_{n}\in \Symp(M, \omega)$ such  that  $S_{\alpha_{n}} h_n = h_n S_{\alpha_{n}} $  
and 
\begin{equation}\label{eq:push_forward}
 \left( t \mapsto h_{n} S_{t} h_{n}^{-1} x \right)_{*} \mathrm{Leb}_{\bbS^1} \in (H_{n-1})_{*}^{-1} (\cU_{n}) , \; \forall x \in M.
\end{equation}
Equation~\eqref{eq:push_forward} may equivalently be restated as 
$$
(H_{n-1} h_{n})_{*} (t \mapsto S_{t}y)_{*} \mathrm{Leb}_{\bbS^1} \in  \cU_{n}, \; \forall y \in M.
$$
We let $H_{n} = H_{n-1} h_{n}$.
Consider the map
$$
\begin{array}{rcl}
\Theta  : \cP(\bbS^1) \times M & \longrightarrow & \cP(M) \\
(\mu,x) & \longmapsto & (H_{n})_{*} (t \mapsto S_{t}x)_{*} \mu.
\end{array}
$$

This mapping is continuous, and since  $ \cP(\bbS^1)$ and $M$ are both compact, it is uniformly continuous.  Now by Equation~\eqref{eq:push_forward},  $\Theta(\Leb_{\bbS^1},x) \in \cU_{n}$ for every $x \in M$. Thus,  there exists a neighborhood $\mathcal{V}$ of $\Leb_{\bbS^1}$ such that  $\Theta(\mathcal{V} \times M) \subset \cU_{n}$.

Let $\alpha_{n+1} = \frac{p_{n+1}}{q_{n+1}}$ be a rational written in irreducible form, and denote by  $\mu_{n} \in \cP(\bb S^1)$ the measure given by the average of the Dirac measures on the orbit of $0 \in \bb S^1$ under the circle rotation by $\alpha_{n+1}$. If $q_{n+1}$ is large enough then the measure $\mu_{n}$
 will be in $\mathcal{V}$ (by convergence of the Riemann sums). Hence we get $\Theta(\mu_n, x) \in \cU_n$ for all $x \in M$.   Note that $\Theta(\mu_n, x)$ is the average of the Dirac measures along the orbit of $H_{n} (x)$ under the map $f_{n} = H_{n} S_{\alpha_{n+1}} H_{n}^{-1}$.  Lastly, we additionaly impose that $q_{n+1}   \in \cQ$.

As explained in Section \ref{sec:general_scheme}, this construction provides a sequence of Hamiltonian diffeomorphisms $f_n=  H_n  S_{\alpha_{n+1}} H_n^{-1}$ that converges, in $C^{\infty}$ topology, to a Hamiltonian diffeomorphism $f$. We will now prove that, by appropriately choosing the sequence $(\alpha_n) =(\frac{p_n}{q_n})$, we can ensure that the limit map $f$ has the desired property:  its set of invariant ergodic measures is exactly $\cE$.

 At the beginning of step $n$ of the construction, we define the set 
$ U(q_{n}) \subset \Ham(M, \omega)$ consisting of all Hamiltonian diffeomorphisms $g$ which have the following property:  for every $x \in M$, the probability measure $
 \frac{1}{q_{n}} \sum_{k=0}^{q_{n}-1} \delta_{g^kx}$ belongs to the set  $\cU_{n-1}.$     This set contains $f_{n-1}$ and is open in the $C^0$-topology.   The numbers $\alpha_{i+1}$, for $i \geq n$, will be chosen such that the Hamiltonian diffeomorphisms $f_i$,  for $i\geq n$, are all contained in a ($C^{\infty}$) neighborhood of $f_n$ whose ($C^{\infty}$) closure is contained in $U( q_{n})$.   This implies that the limit map $f$ will also be contained in $U( q_{n})$.  Of course, the $\alpha_n$'s may be picked such that the map $f$ will have this property for every $n$.  In other words, for every $n$ and every $x \in M$, the probability measure 
 $$
\nu_n := \frac{1}{q_{n}} \sum_{k=0}^{q_{n}-1} \delta_{f^kx}
$$
is contained in $\cU_n$.

We claim  this implies that the set of invariant  ergodic probability measures of $f$ is exactly $\cE$. Indeed, 
to obtain a contradiction, assume that this is not the case:  $f$ has an invariant ergodic probability measure $\mu$ which is not in $\cE$.  
Since the ergodic measures are extremal points in the set of invariant probability measures, the measure $\mu$ does not belong to $\Conv(\cE)$.
This entails the existence of a continuous function $\varphi: M \to \bbR$ which vanishes on $\Fix(S)$ and has the property that $\int \varphi \; d\Vol=0$ but $\int \varphi d\mu \neq 0$.  To see that such $\varphi$ indeed exists, observe that a probability measure $\gamma$ belongs to $\Conv(\cE)$ if and only if it satisfies the following criterion: for every pair of open sets $O,O'$ which are disjoint from $\Fix(S)$ and have the same volume, we must have $\gamma(O)=\gamma(O')$.

By Birkhoff's Ergodic Theorem and the ergodicity of $\mu$, there exists $x \in M$ such that the  Birkhoff means of $\varphi$ converge to $\int \varphi d\mu$.  This, in particular, means that 
the sequence 

$$\int \varphi \; d\nu_n = \frac{1}{q_{n}} \sum_{k=0}^{q_{n}-1} \varphi(f^kx)$$
converges to the non-zero number  $\int \varphi d\mu$.  The contradiction is that, upto passing to a subsequence,  the sequence $\int \varphi d\nu_n$  converges to zero because  $\nu_n \in \cU_n$ and thus must have a weak limit $\nu \in \Conv(\cE)$ and $\varphi$ was picked such that $\int \varphi \; d\nu =0$ for all $\nu \in \Conv(\cE)$.
\end{proof}

\section{Preliminaries}\label{sec:prelim}

The goal of this section is to recall some basic notions of symplectic \& differential geometry as well as proving certain preliminary results which will be used in the proof of Proposition \ref{prop.main}.   

\subsection{Preliminaries on symplectic geometry} \label{sec:prelim_symp}
  Throughout the section $(M, \omega)$  denotes a symplectic manifold.  Recall that a symplectomorphism is a diffeomorphism $\varphi: M \to M$ such that $\varphi^* \omega = \omega$.   The set of all symplectic diffeomorphisms of $M$ is denoted by $\Symp(M, \omega)$.  We will let $\Symp_0(M, \omega)$ denote those elements of  $\Symp(M, \omega)$ which are isotopic to the identity via a compactly supported isotopy. Note that the assumption on compactness of the support of the isotopy is not common.

  Hamiltonian diffeomorphisms constitute an important class of examples of symplectic diffeomorphisms.  These are defined as follows. A smooth, compactly supported, Hamiltonian $H \in C_c^{\infty} ([0,1] \times M)$  gives rise to a time-dependent vector field $X_H$ which is defined via the equation: $\omega(X_H(t), \cdot) = -dH_t$.  The Hamiltonian flow of $H$, denoted by  $\phi^t_H$, is by definition the flow of $X_H$.  A  compactly supported Hamiltonian diffeomorphism is a diffeomorphism which arises as the time-one map of a Hamiltonian flow  generated by a compactly supported Hamiltonian.  The set of all compactly supported Hamiltonian diffeomorphisms is denoted by $\Ham(M, \omega)$.

\medskip 

\noindent {\bf Toric symplectic manifolds.} 
Recall that a  {\bf toric} symplectic manifold is  a closed and connected symplectic manifold $(M^{2n}, \omega)$  equipped with an effective Hamiltonian action of a torus $\bbT^n = \bbS^1 \times \ldots \times  \bbS^1$ whose dimension is half that of $M$, \emph{i.e.} a one-to-one smooth morphism from $\bbT^n$ to $\Ham(M, \omega)$.  Throughout the paper, we identify $\bbS^1$ with $\bbR/\bbZ$ and  $\bb T^n$ with $\bb R^n / \bb Z^n$.

Each of the circle factors in the torus $\bbT^n$ yields a Hamiltonian circle action, or equivalently a periodic Hamiltonian flow with period 1; we will denote by $\mu_1, \ldots, \mu_n \in C^{\infty}(M)$ the Hamiltonians  corresponding to these circle actions.  The {\bf moment map} $\mu : M \rightarrow \bbR^n$ is defined by $x\mapsto(\mu_1(x), \ldots, \mu_n(x))$.  Its image $\Delta := \mu(M)$ is referred to as the {\bf moment polytope} of the action.  According to the convexity theorem of Atiyah and Guillemin-Sternberg  \cite{Atiyah, guillemin-sternberg} $\Delta$ is a convex polytope whose vertices are the images of the fixed points of the action.  Furthermore, $\Delta$  is simple (there are $n$ edges meeting at each vertex), rational (the
edges meeting at a vertex $v$ are of the form $v + t u_i$ where $t\geq 0$ and $u_i \in \bbZ^n$),
and smooth (at each vertex $v$ of $\Delta$ the corresponding $u_1, \ldots, u_n$ may be chosen to be a $\bbZ$--basis of $\bbZ^n$).  

\medskip 

The points of $M$ with non trivial stabiliser are exactly the points which are mapped under $\mu$ to the boundary of $\Delta$. Furthermore, all the  points in the inverse image of a given face of $\Delta$ have the same stabiliser. Thus, we see that there are finitely many stabilisers, each of which is a subgroup of $\bbT^n$ isomorphic to $\bbT^p$, where $p$ is the codimension of the corresponding face of $\Delta$. In particular, we may pick a one-to-one morphism 
\begin{align}\label{eq:locally_free_action}
\begin{split}
\Phi : \bbS^1 \to \bbT^n  \\
t \; \mapsto  tZ,
\end{split}
\end{align}
for some primitive element $Z \in \bbZ^n$, such that the image of $\Phi$ is not included in the union of the stabilisers. Then, the composition of $\Phi$ with the torus action yields a circle action $S$ which is locally free on $M \setminus \Fix(S)$. Such a circle action will be at the basis of our Anosov-Katok construction.

\medskip

It turns out that one can construct a section of the moment map, that is, a continuous map $\sigma : \Delta \rightarrow M$ such that $\mu \circ \sigma = \Id$.  
The fact that the section $\sigma$ exists can be deduced from the proof of Delzant's theorem on classification of toric symplectic manifolds; see for example the construction of the Delzant space $X_{\Delta}$, pages 9 -- 13 in \cite{Guillemin}: in the case of the model $X_{\Delta}$, one can see directly from the construction  that the moment map admits a section;  existence of the section $\sigma$ for arbitrary toric symplectic manifolds then follows from Delzant's theorem that such manifolds are classified by their moment polytopes. Furthermore, the section $\sigma$ is smooth on  $\Int(\Delta)$.

 Given such a section $\sigma : \Delta \rightarrow M$ of the moment map, define the mapping  $\Xi : \Delta \times \bbT^n  \rightarrow M$  by
 \begin{equation}\label{eq:section_moment_polytope}
 (s,t) \mapsto \bbT^n_{t} (\sigma(s) ),
 \end{equation}
 where $\bbT^n_{t} ( \sigma(s) )$ denotes the image of the point $\sigma(s)$ under the action of $t \in \bbT^n$. 
 Endow  $\Int(\Delta) \times \bb T^n$ with the symplectic form 
$$
\frac{1}{2\pi} \Sigma_{i=1}^n d \mu_{i} \wedge d \theta_{i}.
$$
 Then, the mapping $\Xi$ is a symplectomorphism between $\Int(\Delta) \times \bb T^n$ and $\mu^{-1}(\Int(\Delta))$.   This yields a global system of symplectic action-angle coordinates on $\mu^{-1}(\Delta)$ (this is the content of  Remark IV.4.19 on global action-angle coordinates in~\cite{Audin}). Note however that $\Xi$ is not one-to-one on $\Delta \times \bbT^n$, since points on $\mu^{-1}(\partial \Delta)$ have non-trivial stabilisers.
 One can also describe the action in the neighborhood of any degenerate orbit. In particular, there is a local normal form near the orbit of any point $x$ that depends only on the dimension of the face of the moment polytope the interior of which $x$ belongs. In the sequel we will just use the fact that for every face $F$, $\mu^{-1}(F)$ is a submanifold of $M$ (see for example~\cite{Audin}, Proposition IV.4.16). For further details on toric symplectic manifolds, we refer the reader to the books \cite{Audin, Guillemin}. 
\subsection{Preliminary lemmas}\label{sec:prelim_lemmas}
We will be using the following lemmas in the course of proving Proposition~\ref{prop.main}.  In many cases, the lemmas of this section will be applied to a quotient symplectic manifold of the form  $M' = (M \setminus \Fix(S) ) / S_{\frac{1}{q}}$ where $q \in \cQ$; recall that $M'$ is indeed a manifold because picking $q\in \cQ$ guarantees that the action of $S_{\frac{1}{q}}$ on $M\setminus \Fix(S)$ is free.

\medskip

We will be using the following notations throughout the remainder of the paper. 
As before, we 
consider a circle action $S : (x,t) \mapsto S_{t}(x)$ on a symplectic manifold $M$, where $t \in \bbS^1$ and $x \in M$ and we denote the Lebesgue probability measure on the circle by Let $\mathrm{Leb}_{\bbS^1}$.
Given a measurable subset $E$ of $M$ and a point $x$ in $M$, the \emph{time spent by the orbit of $x$ in $E$} is the number
$$
 \mathrm{Leb}_{\bbS^1} (\{t \in \bbS^1: S_{t}(x) \in E \}).
$$
The \emph{vector field tangent to the circle action $S$} is defined as
$$
\vec V(x) = \frac{d}{dt} S_{t}(x)_{\vert t=0}.
$$
The orbits of the action are the integral curves of $\vec V$. Given a diffeomorphism $\Phi$, the action given by 
$S'_{t}(x) = \Phi S_{t} (\Phi^{-1} x)$ is called the conjugated action. The vector field tangent to the conjugated action is the push-forward vector field $\Phi_{*} \vec V$.

\subsubsection*{(a) The transversality lemma}
Given a vector field $\vec V$ on some manifold $M'$, and a submanifold $X$, we wish to remove the tangency points between $\vec V$ and $X$ by performing a small perturbation of $X$, thus obtaining a submanifold which is transverse to every integral curve of the vector field $\vec V$.  Achieving this form of transversality is not in general possible: consider the case where $\vec V$ is a horizontal vector field in $\bbR^3$, with $X$ being the graph of a function on $\bbR^2$; we cannot remove the points on $X$ corresponding to local maxima or minima.  However, the following lemma provides a partial ``fix'' for this situation: it is possible to perform a small perturbation of $X$, or equivalently $\vec V$, so that $X$ becomes  ``as transverse as possible'' to the integral curves of $\vec V$.

We suppose that $X \subset M'$ is a submanifold without boundary (but not necessarily compact\footnote{In the applications of the lemma, the manifold $M'$ will be a quotient of the complement of the fixed points set of our circle action.}).
 Let $K$ be a  compact subset of $X$. We will say that $X$
 is {\it almost transverse} to a vector field $\vec V$ on $K$ if for every integral curve $\gamma$ of $\vec V$, 
every point of $\gamma \cap K$ is isolated in $\gamma \cap X$.
 We will say that $X$ is \emph{stably almost transverse} to $\vec V$ on $K$ if this property holds not only for $\vec V$ but also for  $\Phi_{*}\vec V$ for every $C^\infty$ diffeomorphism $\Phi$ of $M'$ in some neighborhood of the identity.

\begin{lemma}
\label{lemma.transversality}
Consider a symplectic manifold $M'$ with a non-vanishing vector field $\vec V$.  Let $X$ be a submanifold of $M'$ of codimension at least $1$, without boundary, and $K$ a compact subset of $X$.  For any $C^{\infty}$ neighborhood of the identity  $\mathcal W \subset \Symp(M, \omega)$ and any open neighborhood $O$ of $X$, there exists $\Phi \in \mathcal{W}$, whose support is contained in $O$, such that $X$ is stably almost transverse to $\Phi_{*}\vec V$ on $K$.
%
%
%
%
\end{lemma}


\subsubsection*{(b) The thickening lemma}

\begin{lemma}  \label{lemma.thickening}
Consider a metric space $M'$ with a continuous circle action $S$.   Let $X$ be a compact subset of $M'$, and assume that no orbit of the action spends a positive amount of time in $X$.

Then, for every $\varepsilon>0$ there exists $\delta>0$ such that no orbit of the action spends more time than $\varepsilon$ in the $\delta$-neighborhood of $X$.
\end{lemma}

\subsubsection*{(c) The stability lemma}

\begin{lemma}\label{lemm:stability}
Consider a symplectic manifold $M'$ with a locally free circle action $S$. 
Let $K_{1}, K_{2}$ be compact subsets of $M'$ with $K_{1} \subset \Int(K_{2})$. 

Let $c \in (0,1)$ and assume that every orbit of the action $S$ spends more time than $c$ in $K_{1}$.  Then, there exists a $C^0$ neighborhood $\mathcal W$ of $\Id  \in \Diff(M)$ such that every orbit of the conjugated action $\Phi S \Phi^{-1}$ spends more time than $c$ in $K_{2}$.
\end{lemma}

\subsubsection*{(d) The transportation lemmas}
Consider $\mathbb R^{2n}$ equipped with the coordinates $x_1, y_1, \ldots, x_n, y_n$.  A standard  polydisc in $\mathbb R^{2n}$ is a subset of the form 
$$\prod_{i=1}^n [a_i, b_i] \times [c_i, d_i] := \{(x_i, y_i): x_i \in [a_i, b_i], y_i \in [c_i, d_i] \}.$$  
By a \emph{polydisc} in $M$ we mean the image of a symplectic embedding of a standard polydisc.  Note that our polydiscs are all closed.

In the next two lemmas $P$ denotes the polydisc  $[-a,a]^{2n} \subset \bb R^{2n}$ for some $a>0$.  We assume that $P$ is equipped with the standard symplectic structure it inherits from $\bb R^{2n}$. 

\medskip

\begin{lemma}\label{lemm:transport_lemma} Let $\phi_1, \phi_{2}:P \rightarrow M''$ be two symplectic embeddings of $P$ into a symplectic manifold $(M'', \omega)$.  There exists $\delta_{0}>0$ and a compactly supported $\Psi \in \Symp_{0}(M'')$ such that 
$$
\Psi \circ \phi_{1} = \phi_{2} \text{ on } [-\delta_{0},\delta_{0}]^{2n}. $$
Furthermore, if $\phi_{1}(0) \neq \phi_{2}(0)$, then  we may require, in addition to the above, that $$
\Psi \circ \phi_{2} = \phi_{1} \text{ on } [-\delta_{0},\delta_{0}]^{2n}.
$$

%
%
\end{lemma}

\medskip

In the next lemma $P_1, P_2 \subset \bb R^{2n}$ denote the polydiscs $p_1 + [-b,b]^{2n}$, $p_2 + [-b,b]^{2n}$, respectively, where $p_1, p_2$ are points in $\bb R^{2n}$.  The polydisc $P$ is as in the above lemma. The following statement is well-known in the two dimensional setting when $n=1$;  we leave it to the reader to check that the proof can be reduced to the case where $n=1$.

\begin{lemma}\label{lemm:polydisc_swapping_lemma}
Suppose that $P_1, P_2$ are disjoint and are contained in the interior of $P$.  There exists a symplectomorphism $\Psi$, whose support is compactly contained in the interior of $P$, such that 

$$\Psi(P_1) = P_2  \;\; \& \;\; \Psi(P_2) = P_1.$$ 
%
%
\end{lemma}

\subsection{Proofs of the lemmas} 

In this section, we will provide proofs for Lemmas \ref{lemma.transversality}, \ref{lemma.thickening}, \ref{lemm:stability}, \ref{lemm:transport_lemma}  of the previous section, leaving the proof of Lemma \ref{lemm:polydisc_swapping_lemma}  to the reader.

\subsubsection*{(a) Proof of Lemma \ref{lemma.transversality}: Transversality Lemma}
We will need the following notion in the course of the proof. Let $M'$ be as in the statement of the lemma and $X$ be a hypersurface without boundary in $M'$.
Given a vector field $\vec V$ and a positive integer $k$, we say that $X$ {\it has a contact of order $k$ with $\vec V$} at some point $x \in X$ if there exist an integral curve $\alpha: (-\varepsilon,\varepsilon) \to M'$ of $\vec V$ and a smooth curve $\beta:(-\varepsilon,\varepsilon) \to X$ such that $\alpha(0)=\beta(0)=x$ and the derivatives of $\alpha$ and $\beta$ coincide at $0$ up to order $k$. 

\begin{remarks}\label{rem:contact} Here are some useful observations about the above notion.
\begin{enumerate}
\item If $\Phi$ is a diffeomorphism, then $\Phi(X)$ has a contact of order $k$ with $\vec V$ if and only if $X$ has a contact of order $k$ with $\Phi_{*}^{-1} \vec V$.

\item If $\alpha:I \to M'$ is any integral curve of $\vec V$ and if the intersection set
$$
\alpha(I) \cap X
$$
has an accumulation point at $z=\alpha(t)$, then $X$ has a contact of every positive order with $\vec V$ at $z$.

\item If $K$ is a compact subset of $X$, then the set of diffeomorphisms $\Phi$ such that $X$ has no contact of order $k$ with $\Phi_{*} \vec V$ at points of $K$ is open in the space of diffeomorphisms equipped with the $C^{\infty}$ topology.

\item If $X$ has no contact of order $k$ with $\vec V$ on $K$, for some positive integer $k$ and some compact subset $K$ of $X$, then $X$ is stably almost transverse to $\vec V$ on $K$. This is a direct consequence of the previous two points.
\end{enumerate}
\end{remarks}

  Our proof of Lemma \ref{lemma.transversality} requires the  claim below. We assume the hypotheses of Lemma~\ref{lemma.transversality}. 
  
\begin{claim} \label{cl:transversality}
Consider a symplectic manifold $M'$, a submanifold $X \subset M'$ without boundary and of codimension at least $1$, and some neighborhood $O$ of $X$.
Let $k \geq \dim(M')$.

Then, every point $z_{0} \in X$ admits a pair of open neighborhoods $(W,U)$ with $W \subset \overline{W} \subset U$ in $M'$, which are contained in $O$ and have the following property. 
For every non-vanishing vector field $\vec V$, and every neighbourhood $\mathcal N$ of the identity in $\Symp(M')$,   there exists $\Phi \in \mathcal N$, compactly supported in $U$, such that $X$ has no contact of order $k$ with $\Phi_{*}\vec V$ in $\overline{W}$.
\end{claim}

Before proving Claim \ref{cl:transversality}, we will show that it implies Lemma \ref{lemma.transversality}.

\begin{proof}[Proof of the transversality lemma]
Observe that it is sufficient to find $\Phi \in \mathcal{W}$, with support in $O$, such that $\Phi(X)$ has no contact of order $k = \mathrm{dim}(M')$ with $\vec{V}$ at any point of $K$; see Remark \ref{rem:contact}.

For each point $z \in K$, let $W_z, U_z$ be open neighborhoods of $z$  in $M'$ as provided by Claim \ref{cl:transversality}.  The collections $W_z, U_z$ are two coverings of $K$ by open subsets of $M'$. By passing to subcovers we may assume these coverings are finite and we may denote them by $W_i, U_i$, where $ i \in \{1, \dots, N\}$

A first application of the claim provides $\Phi_{1} \in \mathcal{W}$, with support in $U_1 \subset O$, such that $X$ has no contact of order $k$ with $\Phi_{1*}\vec V$ on $\overline{W}_{1}$.
A second application provides $\Phi_{2} \in \Symp(M')$ such that $X$ has no contact of order $k$ with $(\Phi_2 \circ \Phi_1)_{*}\vec V$ on $\overline{W}_{2}$. Furthermore, since $\Phi_{2}$ may be chosen to be arbitrarily $C^{\infty}$ close to the identity, by point 3 of Remark~\ref{rem:contact} we can ensure that $\Phi_2 \circ \Phi_1 \in \mathcal{W}$ and that $X$ has no contact of order $k$ with $(\Phi_2 \circ \Phi_1)_{*}\vec V$ on $\overline{W}_{1}$ either. Note also that $\Phi_2 \circ \Phi_1$ is supported in $U_1 \cup U_2 \subset O$. 

We proceed to build a sequence $(\Phi_{i})_{i=1, \dots, N}$ in $\Symp(M')$ with analogous properties. In particular, the map $\Phi = \Phi_{N} \circ \cdots \circ \Phi_{1}$ belongs to $\mathcal{W}$,  is supported in $U_1 \cup \ldots \cup  U_N \subset O$, and  $X$ has no contact of order $k$ with $\Phi_{*}\vec{V}$ on $\overline{W}_{1} \cup \ldots \cup \overline{W}_N$, which contains $K$. By point 4 of Remark~\ref{rem:contact}, $X$ is stably almost transverse to 
$\Phi_{*}\vec{V}$ on $K$, as wanted.
\end{proof}

\begin{proof}[Proof of Claim \ref{cl:transversality}]
Since every submanifold is contained in an open hypersurface without boundary, we will restrict ourselves to the case where $X$ is a hyper surface.  Take $(U, x_1, y_1, x_2, y_2, \ldots, x_n, y_n)$ to be a Darboux chart centered at $z_0$, and contained in $O$, in which $X$ is given by the equation $x_n = 0$.  By working in these local coordinates we may assume that $M'= \bb R^{2n}$, $X$ is the hyperplane given by $x_n= 0$, and the point $z_0 = 0$, where $0$ is the origin in $\bb R^{2n}$.  We pick $W$ to be any open neighborhood of the origin such that $\overline{W}$ is contained in $U$.

Denote $\mathrm{dim}(M') = 2n$.  To prove the claim, we will translate the property of having contact of order $k$ to an equivalent property in the space $J^k(\bbR^{2n},\bbR)$ of $k$-jets of maps from $\bbR^{2n}$ to $\bbR$. Recall that $J^k(\bbR^{2n},\bbR)$ identifies with $\bbR^{2n} \times \mathrm{Pol}_{k}$, where $\mathrm{Pol}_{k}$ is the linear space of polynomials of degree at most $k$ on $\bbR^{2n}$; see for example Section 1.1 of \cite{Eliashberg-Mishachev}. 

We now consider a non vanishing vector field $\vec V$ on $\bbR^{2n}$.
Let $\Sigma_{\vec V}$ be the set of couples $(z, P) \in J^k(\bbR^{2n},\bbR)$ such that there exists some integral curve $\gamma: (-\varepsilon,\varepsilon) \to \bbR^{2n}$ of the vector field $\vec V$ with $\gamma(0)=z$ and $\lim_{t\to 0} \frac{P(\gamma(t))}{t^k}= 0$. Note that  $\Sigma_{\vec V}$ is a submanifold of $J^k(\bbR^{2n},\bbR)$ of codimension $k+1$. Indeed, to verify this particular point we can work in a small neighborhood of the point $z \in \bb R^{2n}$ where we may assume that $\vec{V}$ coincides with  the vector field~$\frac{\partial}{\partial x_1}$.  Then $\Sigma_{\vec V}$ may be described locally as $$
\left\{(z,P): P(0)=0, \frac{\partial P}{\partial x_{1}}(0)=0, \dots, \frac{\partial ^k P}{\partial x_{1}}(0)=0 \right\},
$$
which is a linear subspace of $J^k(\bbR^{2n},\bbR)$ of codimension $k+1$.

Finally, here is the promised translation. Let $f:\bbR^{2n} \to \bbR$ be a smooth  function and suppose that $f^{-1}(0)$ is a hypersurface.  Then,   $j^k(f)_{z}$, the $k$-jet  of $f$ at $z$, belongs to  $\Sigma_{\vec V}$  if and only if $f^{-1}(0)$ has a contact of order $k$ with $\vec V$ at $z$.

For the rest of the proof we set $f=x_n$ so that $X = f^{-1}(0)$.  We claim that we can find a smooth function $H : \bbR^{2n} \to \bbR$, arbitrarily $C^{\infty}$ close to $0$, such that under  the map  $\bbR^{2n} \to J^k(\bbR^{2n},\bbR)$, 
$$
z \mapsto j^k(f \circ \phi^1_{H})_{z},
$$
the image of $U$ is disjoint from $\Sigma_{\vec V}$.

Let us first suppose that such $H$ exists and explain how this leads to the conclusion.  Note that $(f \circ \phi^1_{H})^{-1}(0) = \phi^{-1}_H(X)$ and so we have that $ \phi^{-1}_H(X)$ has no contact of order $k$ with $\vec V$ inside the set $U$.  Since  $H$ can be picked to be arbitrarily $C^{\infty}$ close to $0$, we may perform a cut-off to obtain a function $G$, which is compactly supported in $U$ and is still $C^\infty$ close to zero, such that  $ \phi^{-1}_G(X)$ has no contact of order $k$ with $\vec V$ inside $\overline{W}$.  

It remains to explain how to find $H$. Consider the map $$\Psi : \overline{U} \times \mathrm{Pol}_{k+1}  \to J^k( \overline{U} ,\bbR)$$
$$
(z, H) \mapsto j^k(f \circ \phi^1_{H})_{z}.
$$
We will check below that there exists some  neighborhood  $\mathcal U$ of $0$ in $\mathrm{Pol}_{k+1}$ such that $\Psi$ is a submersion at every point $(z, H) \in \overline U \times \mathcal U$.

This entails that the restriction of $\Psi$ to  $\overline U  \times \mathcal{U} $ is transverse to any submanifold; in particular it is transverse to $\Sigma_{\vec V}$.  Next, applying the parametric transversality theorem, we obtain a dense subset of polynomials $\mathcal H \subset \mathcal{U}$ such that for any fixed $H \in \mathcal H$ the mapping
$$\Psi(\cdot, H) : \overline U   \to J^k(\overline{U} ,\bbR)$$
$$
z \mapsto j^k(f \circ \phi^1_{H})_{z},
$$
is transverse to $\Sigma_{\vec V}$. Since $k\geq 2n$, we get $\dim(\bbR^{2n}) < k+1 =  \codim(\Sigma_{\vec V})$. Then, transversality implies that the image of $\overline{U}$ under $z \mapsto j^k(f \circ \phi^1_{H})_{z}$ is actually disjoint from $\Sigma_{\vec V}$, which is the desired property.

 It remains to check that $\Psi$ is a submersion for all $H$ in some neighborhood of $0$ in $\mathrm{Pol}_{k+1}$.  Since being a submersion is an open property, and $\overline U$ is compact, it is enough to check that the mapping $\Psi$ is a submersion when $H=0$.  In the computations below we will identify   $J^k(\overline{U},\bb R) = \overline{U} \times \mathrm{Pol}_{k}$. Furthermore, we will also identify  the tangent spaces to points of $\overline{U} \times \mathrm{Pol}_{k}$ and $\overline{U} \times \mathrm{Pol}_{k+1},$ with $\bb R^{2n} \times \mathrm{Pol}_{k}$ and $\bb R^{2n} \times \mathrm{Pol}_{k+1}$, respectively.
 
  We would like to compute $D \Psi_{(z,0)} (u, G)$ where $u$ is a vector in $\overline{U} $ and $G \in \mathrm{Pol}_{k+1}$.  It is easy to see that $D \Psi_{(z,0)} (u, 0) = (u,0)$.  Hence, we will only consider the case where $u=0$.  Now, we compute: 
  $$D \Psi_{(z,0)} (0, G) = \frac{\partial}{\partial t}|_{t=0}\; \Psi(z, tG) = \frac{\partial}{\partial t}|_{t=0}\; j^k(f \circ \phi^t_{G})_{z}$$ 
  
  $$ = j^k(\frac{\partial}{\partial t}|_{t=0}\; f \circ \phi^t_{G})_{z} = j^k(\{f, G\})_{z},$$
  where we have used the fact that $\frac{\partial}{\partial t}|_{t=0}\; f \circ \phi^t_{G} = \{f, G\}$.  Here,  $\{f, G \}$ denotes the Poisson bracket of $f$ and $G$ and it is given by $\{f, G \} = \sum_{i=1}^{2n} \frac{\partial f}{\partial x_i} \frac{\partial G}{\partial y_i} - \frac{\partial f}{\partial y_i} \frac{\partial G}{\partial x_i}$.  Recall that we picked $f=x_n$ and so $\{f, G \} = \frac{\partial G}{\partial y_n}$.  Hence, we have established 
   $$D \Psi_{(z,0)} (0, G) =  j^k \left( \frac{\partial G}{\partial y_n} \right)_{z}.$$ 
   
   Now, for any given polynomial $P \in \mathrm{Pol}_{k}$ let $G = y_n P$ and note that $j^k \left( \frac{\partial G}{\partial y_n} \right)_{z} = P.$  We conclude that $D \Psi_{(z,0)}$ is surjective, for any $z \in \overline{U}$, and so $\Psi$ is a submersion at $(z, 0)$ for all $z \in \overline{U}$.  
\end{proof}

\subsubsection*{(b) Proof of Lemma \ref{lemma.thickening}: Thickening Lemma}
As before, we will denote by $t \mapsto S_{t}$ the action of the circle by homeomorphisms on $M$. Fix $\varepsilon > 0$ and let $x \in X$.  Since $
 \mathrm{Leb}_{\bbS^1} (\{t \in \bbS^1: S_{t}(x) \in X \}) = 0,$ we can find a compact subset $I_{x} \subset \bb S^1$, whose Lebesgue measure is more than $1-\varepsilon$, and such that for every $t \in I_{x}$, we have $S_{t}(x) \not \in X$.   Compactness of $I_{x}$ and $X$ implies that there is some $\delta_{x}>0$ such that $d(S_{t}(x), X) > \delta_{x}$ for every $t \in I_{x}$.   Observe that for every $x \in X$, we can find a neighborhood $V_{x}$ of $x$ in $M'$ such that we still have $d(S_{t}(y), X) > \delta_{x}$ for $t \in I_{x}$ and $y \in V_{x}$.  Since $X$ is compact, we can find $V_{x_1}, \ldots, V_{x_m}$ such that the finite union $V:= \cup_i V_{x_i}$ covers $X$.  Pick $0 < \delta$,  less than each $\delta_{x_{i}}$, such that $V_{\delta}(X)$, the $\delta$-neighborhood of $X$, is contained in $V$.  Observe that we have proven the following statement: for every $y \in V$ 
  $$\mathrm{Leb}(\{t \in \bb S^1:  S_t(y) \in V_{\delta}(X)\}) < \varepsilon  .$$
  
The above inequality also holds for any point $y$ whose orbit under the action meets $V$, since its orbit coincides with the orbit of a point in $V$.  On the other hand, if $y$ is a point in $M'$ whose orbit does not meet $V$, then $$\mathrm{Leb}(\{t \in \bb S^1:  S_t(y) \in V_{\delta}(X)) = 0.$$
This completes the proof.


\subsubsection*{(c) Proof of Lemma \ref{lemm:stability}: Stability Lemma}
If $\Phi$ is uniformly close to the identity then the orbits of the conjugated action are uniformly close to the orbits of the action $S$. On the other hand the hypotheses entail that the distance from $K_{1}$ to the complement of $K_{2}$ is positive. The lemma follows.

\subsubsection*{(d) Proof of Lemma \ref{lemm:transport_lemma}: Transportation Lemma}
We begin by proving the first statement of the lemma, that is there exist $\Psi \in \Symp_0(M'',\omega)$ and $\delta_0 >0$ such that $\Psi \phi_{1} = \phi_{2}$ on $[-\delta_0,\delta_{0}]^{2n}$.  

Note that since $\Symp_{0}(M'')$ acts transitively on $M''$ we can suppose that $\phi_{1}(0)=\phi_{2}(0)$. Let $U$ be a Darboux chart centered at $\phi_{1}(0)=\phi_{2}(0)$. By replacing $P$ with a smaller polytope, we may suppose that the image of $P$ under $\phi_1, \phi_2$ is contained in the Darboux chart $U$.  This allows us to reduce the problem to the following setting:  $M'' = \bb R^{2n}$ and $\phi_1, \phi_2$ are symplectic embeddings of a polydisc $P = [-a, a]^{2n}$ such that $\phi_1(0) = \phi_2(0) = 0$. 
We must find $\delta_0 >0$ and a symplectic isotopy $\Psi^t$, which is compactly supported in $U$, with the property that  $\Psi^0 = \Id$ and $\Psi^1 = \phi_2 \phi_1^{-1}$ on $[-\delta_0,\delta_{0}]^{2n}$.

Pick any $0 <b <a$ and let $P' = [-b, b]^{2n}$.  By the well-known ``extension after restriction principle" \cite{Ekeland-Hofer}, we can find a compactly supported symplectomorphism $\psi$ of $\bb R^{2n}$ such that $\psi(0) = 0$ and $\psi|_{P'} = \phi_2 \phi_1^{-1}|_{P'}$.  There exists a symplectic isotopy $(\Psi^t)_{t \in [0,1]}$, compactly supported in $\bbR^{2n}$, such that $\Psi^{0} = \mathrm{Id}$, 
$\Psi^t(0)=0$ for every $t \in [0,1]$,
and $\Psi^1 = \psi$; for an explanation see~\cite{Schlenk}, proof of Proposition~1.7. There exists $\delta_0 < b$  such that 
$$
K := \bigcup_{t \in [0,1]} \Psi^t ([- \delta_0 ,\delta_0]^{2n}) \subset U.
$$

Let $H$ be the Hamiltonian whose flow is $\Psi^t$.  By cutting off the Hamiltonian $H$ in a neighborhood of the set $K$, we obtain a new Hamiltonian $G$ which is supported in $U$ and has the property that $\phi^1_G = \psi$ on  $[- \delta_0 ,\delta_0]^{2n}$.  We set $\Psi = \phi^1_G$.  This completes the proof of the first part of the lemma.

For the second part of Lemma \ref{lemm:transport_lemma}, note that we can find a compactly supported symplectomorphism $\theta \in \Symp_{0}(M'')$ which exchanges $\phi_{1}(0)$ and $\phi_{2}(0)$. Now,  if we apply the above construction, which is local, independently once near $\phi_{1}(0)$ and another time near $\phi_{2}(0)$ we will obtain $\Psi_1, \Psi_2$ supported near $\phi_{1}(0), \phi_{2}(0)$, respectively, such that $\Psi_1 \theta \phi_2 = \phi_1$ and $\Psi_2 \theta \phi_1 = \phi_2$ on $[-\delta_0, \delta_0]^{2n}$ for some $\delta_0 > 0$.  We let $\Psi = \Psi_1 \Psi_2 \theta$.
\section{Proof of Proposition~\ref{prop.main}}\label{sec:proof_main_prop}

For the rest of the paper $(M, \omega)$ will denote  a closed toric symplectic manifold with moment map $\mu : M \rightarrow \Delta$.  We fix a locally free Hamiltonian circle action $S$ obtained as described via Equation \eqref{eq:locally_free_action}.
We consider an integer $q \in \cQ$ and a small positive $\varepsilon \in \bbR$.   We fix a Riemannian metric on $M$ and denote
$$
B_{\varepsilon}(\Fix(S)) = \bigcup_{x \in \Fix(S)}B_{\varepsilon}(x),
$$
where $B_{\varepsilon}(x)$ is the ball of radius $\varepsilon$ around $x$.

\medskip

Before giving the proof of Proposition \ref{prop.main} in detail, we outline here the main ideas of the proof.  Our goal is to construct a Hamiltonian diffeomorphism $h$ satisfying the three requirements of Proposition \ref{prop.main}.  To ensure that the first two properties are satisfied, we will carry out the construction on a compact subset of the quotient $(M \setminus \Fix(S)) / S_{\frac{1}{q}}$.  Recall that $q$ being in $\cQ$ guarantees that this quotient is a smooth symplectic manifold; see Equation \eqref{eq:def_cQ}.

The difficult part of our task is to ensure that the third property in Proposition \ref{prop.main} is satisfied. To achieve this,  we will first construct, in Lemma \ref{lemm.equidistributionboxes}, a collection of closed sets denoted by $A_1, \ldots, A_N$ satisfying the following equidistribution property: 
\begin{enumerate}
\item[] \emph{The sets $A_{i}$ have diameters  less than $\varepsilon$, their interiors are disjoint, their volumes are equal, their boundaries are of zero volume and 
$$
 \Fix(S) \subset W: = \left( M \setminus \bigcup_{i} A_{i} \right) \subset B_{\varepsilon}(\Fix(S)).
$$}
\end{enumerate}
We will refer to the $A_i$'s as the {\it equidistribution boxes}.  Then, we will construct $h$ such that for each  $x \in M$, the orbit $hS_th^{-1}(x)$ under the conjugated action spends roughly the same amount of time in each of the $A_i$'s: for all $i,j$ we will have $$ \mathrm{Leb}(\{t\in \bbS^1: h S_{t} h^{-1}(x) \in A_i\}) \approx \mathrm{Leb} (\{t\in \bbS^1: h S_{t} h^{-1}(x) \in A_j\}).$$
 A more precise version of the above statement, along with the proof of the fact that it implies the 3rd item in Proposition \ref{prop.main}, is given  in Proposition \ref{prop.main2}.

\smallskip 

To construct $h$ we fix $\varepsilon'>0$ and we will build a collection of disjoint, symplectomorphic polydiscs $c_k$ with the following two critical properties:
\begin{enumerate}
\item[i.] There exists a $C^{\infty}$-small symplectomorphism $\Psi$, commuting with $S_{\frac{1}{q}}$, such that for each  $x \in M$ its orbit $\Psi S_t \Psi^{-1}(x)$, under the conjugated action  will spend more time than $1-\varepsilon'$  in $W \cup_{k} c_{k} $: 
 $$ \mathrm{Leb}(\{t\in \bbS^1: \Psi S_{t} \Psi^{-1}(x) \in W \cup_{k} c_{k}  \}) > 1 -\varepsilon'.$$

\item[ii.]  There exists a symplectomorphism $\Theta$, commuting with $S_{\frac{1}{q}}$, which equidistributes the small boxes $c_k$ among the equidistribution boxes $A_i$.
\end{enumerate}
The symplectomorphism $h$ will then be the composition $\Theta \circ \Psi$.  The construction of the small boxes and the symplectomorphism $\Psi$ is carried out in Lemma \ref{lemm.smallboxes}.  The symplectomorphism $\Theta$ is constructed in Claim \ref{cl:distributing_small_boxes} in  the course of the proof of Proposition \ref{prop.main2}.

\smallskip

We should mention that, in general, there exist symplectic obstructions to finding a symplectomorphism equidistributing a given collection of polydiscs; see for example the symplectic camel problem in Section 1.2 of \cite{McDuff-Salamon}.  To avoid such obstruction, the small boxes $c_k$ must be picked to be sufficiently small.  The details of the construction requires the introduction of a third collection of polydiscs $B_j$ of intermediate size which will be referred to as the transportation boxes; see Lemma \ref{lemm.bigboxes}.

\subsection{Equidistribution boxes} \label{sec:equidistribution}
The goal of this section is to construct the equidistribution boxes mentioned above and prove that they satisfy the properties stated in the lemma below.  The mapping $\Xi : \Delta \times \bbT^n  \rightarrow M$ in the statement below is as in Equation \eqref{eq:section_moment_polytope}.  By saying that a map $s$ \emph{acts as a permutation by $k$-cycles} on a set $E$ we mean that $s$ permutes the elements of $E$, and every orbit of this permutation has cardinality $k$.

\begin{lemma}[Equidistribution Boxes]
\label{lemm.equidistributionboxes}
Let $q \in \cQ$ and $\varepsilon>0$.
There exist $N \in \cQ$ and closed subsets $A_{1}, \dots, A_{N}$ of $M\setminus\Fix(S)$ such that:

\begin{enumerate}
\item (Equidistribution Property) The sets $A_{i}$ have diameters  less than $\varepsilon$, their interiors are disjoint, their volumes are equal, their boundaries are of zero volume, and 
$$
W: = \left( M \setminus \bigcup_{i} A_{i} \right) \ \subset \ B_{\varepsilon}(\Fix(S)).
$$

\item $\bigcup_{i} A_{i}$ is invariant under the circle action $S_t$.

\item $S_{\frac{1}{q}}$ acts on the set $\{A_1, \ldots, A_N\} $ as a permutation by $q$-cycles.

\item (Action-Angle Coordinates)  For each $A_{i}$, we have $$ A_i = \Xi(P \times T),$$ where $\Xi$ is the map defined in Equation \eqref{eq:section_moment_polytope}, $P$ is some polytope included in the moment polytope $\Delta$, and $T$ is a cube in the torus $\bbT^n$ obtained from a subdivision of $\bbT^n$ into equal cubes.  The map $\Xi$ defines a symplectomorphism from $\Int(P\times T)$ to $\Int(A_i)$. 
 
%
%

%
%
%
\end{enumerate}
\end{lemma}

\begin{proof}

To construct the $A_i$'s, we begin by subdividing  the moment polytope $\Delta$ into a collection of polytopes  of small diameter and equal volume.  By subdividing we mean that  two distinct polytopes can only intersect at lower dimensional faces and their union covers $\Delta$.  Now, let $\mathcal{P} = \{P\}$ be the set consisting of those polytopes  from the subdivision which do not contain any of the vertices of $\Delta$.    Let $N''$ denote the total number of the polytopes in $\mathcal{P}$. The subdivision may be carried in such a way that $N''$ is relatively prime  to $q_{0}$ and so $N'' \in \cQ$ (see equation~\ref{eq:def_cQ}).

We consider a decomposition of the torus, which is invariant under translation by $S_{\frac{1}{q}}$, into equal cubes $T$.   More precisely, we obtain these cubes by subdividing each $\bb S^1$ factor of $\bb T^n$ into $N'q$ subintervals of equal length, where $N'$ is picked to be relatively prime to $q_{0}$. 
Hence, the cubes $T$ are all cubes in $\bb T^n$ of the form 
$$v+ \left[0, \frac{1}{N'q}\right]^n, $$  
where $v \in \frac{1}{N'q} \bb Z^n$.
Note that the total number of cubes $T$ is $(N'q)^{n}$.    The fact that this decomposition of $\bb T^n$ into the cubes $T$ is invariant under translation by $S_{\frac{1}{q}}$ follows from Equation \eqref{eq:locally_free_action}.\begin{figure}[h]
\centering
\def\svgwidth{\textwidth}
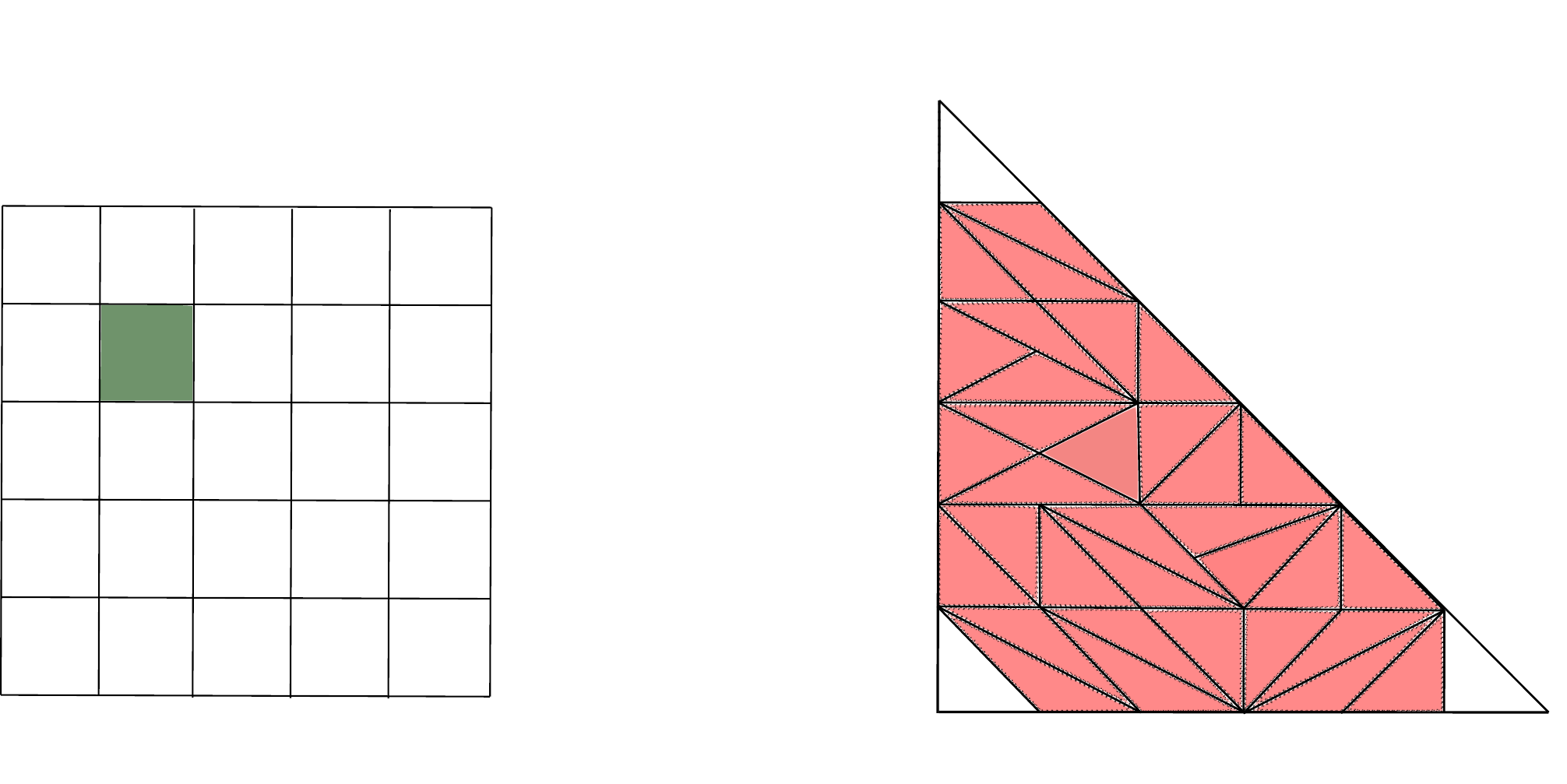
\caption{ Depiction of the equidistribution boxes $A_i$ in the case of $\mathbb{C}P^2$:  On the left, a subdivision of $\mathbb T^2$ into small equal cubes $T$ with one such cube colored in green.  On the right, a sample subdivision of the moment polytope of $\mathbb C P^2$ into small poltytopes of equal volume.  The collection $\mathcal{P} = \{P\}$ consists of the polytopes colored in pink.    The sets $A_i$ have equal volumes. }
\end{figure}

Finally, we obtain the $A_i$'s by considering  the images under $\Xi$ of all the products $P \times T$ of the polytopes $P \in \mathcal{P}$ and the cubes $T$.   Observe that the total number of the $A_i$'s is $N= N'' (N'q)^n$ which belongs to $\cal Q$. Also note that, since $\Xi$ is a symplectomorphism, the volume of each $A_{i}$ equals the product of the Haar volume of $T$ with the standard volume of $P$, and thus they are all equal.

By picking the subdivisions of $\Delta$ and $\bb T^n$ into polytopes and cubes, respectively, to be sufficiently fine we can ensure that the $A_i$'s are of diameter less than $\varepsilon$ and that $W$ is contained $ B_{\varepsilon}(\Fix(S))$ (for this it is crucial that our global section $\sigma$ is defined on $\Delta$ and not only on $\mathrm{Int} (\Delta)$).
 It is not difficult to check, from the construction, that the $A_i$'s satisfy all the remaining properties.
\end{proof}

The following observation will be used below. As we recalled in Section \ref{sec:prelim}, the inverse images of the faces of the moment polytope are submanifolds of $M$.  Thus, by point 4 of the lemma, each face of each $A_{i}$ is compactly included in an open submanifold of $M$.

\subsection{Transportation boxes}  \label{sec:transportation}
Let  $B' \subset B$ be two polydiscs in $M$.   We say $B' $ is a \emph{sub-polydisc} of $B$,  of the form   
$$\prod_{i=1}^n [a_i', b_i'] \times [c_i', d_i'] \subset \prod_{i=1}^n [a_i, b_i] \times [c_i, d_i],$$
 if there exists a symplectic embedding which maps $ \Pi_{i=1}^n [a_i, b_i] \times [c_i, d_i]$ to $B$ and $\Pi_{i=1}^n [a_i', b_i'] \times [c_i', d_i']$ to $B'$. 
In the statement below, $A_{1}, \dots, A_{N}$ are the equidistribution boxes given by Lemma \ref{lemm.equidistributionboxes} and   $W = M \setminus \cup A_{i} $; in particular $N$ denotes the number of equidistribution boxes (these data depend on $q\in \cQ$ and $\varepsilon>0$).

\medskip

Roughly speaking, the aim of the following lemma is to cover most of the interior of the $A_{i}$'s by disjoint polydiscs $B'_{j}$ so that

\begin{enumerate}
\item The $B'_{j}$'s are symplectomorphic, and each $A_{j}$ contains the same number of $B'_{j}$'s,
\item The collection $\{B_{j}'\}$ is invariant under $S_{\frac{1}{Nq}}$, 
\item There exists a slightly perturbed circle action each of whose orbits spends most of its time inside the union of these polydiscs.
\end{enumerate}

The number $\varepsilon$ was used to ensure that the boxes $A_{i}$ are well distributed in $M$. We will use a different number, denoted by $\varepsilon'$, to ensure that every orbit is well distributed among the $A_{i}$'s (this will become clear in Proposition~\ref{prop.main2}).

\begin{lemma}[Transportation Boxes]
\label{lemm.bigboxes}
Let $q \in \cQ$ and $\varepsilon>0$ be as in Lemma~\ref{lemm.equidistributionboxes}, and let $\varepsilon'>0$.
There exist two families of polydiscs $ B_1' \subset B_1,  \ldots, B_{N_1}' \subset B_{N_1}$,  and a map  $\Psi_1 \in \Symp_0(M, \omega)$ such that the following properties are satisfied:

\begin{enumerate}

\item The polydiscs $B_i$ are all symplectomorphic to the standard polydisc $[-r, r]^{2n}$ for some $r>0$, have disjoint interiors, each $B_i$ is included in the interior of some $A_j$, and each $A_j$ contains the same number of $B_i$'s,

\item $B_i'$ is a sub-polydisc of $B_i$ of the form   $[-r', r']^{2n} \subset [-r, r]^{2n}$, where $r' < r$,

\item  $S_{\frac{1}{Nq}}$ acts on the $B_{i}$'s and  on the $B_{i}'$'s  as a permutation by  $Nq$-cycles, 

\item $\Psi_1$ is $C^{\infty}$--close to the identity, its support is disjoint from $\Fix(S)$, and it commutes with $S_{\frac{1}{Nq}}$,

\item There exists a compact subset $K$ of $\Int\left(\cup_{i} B_{i}'\right)$ such that every orbit of the conjugated circle action $\Psi_{1} S \Psi_{1}^{-1}$  spends more time than $1-\varepsilon'$  in $K \cup W$.
\end{enumerate}
\end{lemma}

The families of polydiscs $B_i, B_i'$ will be referred to as the \emph{transportation boxes.} 

\begin{proof}
The construction of $\Psi_1$ and $B_i$'s will be  mostly carried out in the quotient $$M' = (M \setminus \Fix(S) ) / S_{\frac{1}{Nq}}.$$
Recall, from the explanation after Equation \eqref{eq:def_cQ}, that $M'$ is a manifold and, moreover, it  naturally inherits the symplectic structure  and the Hamiltonian circle action of $M$.  We will denote the symplectic form and the circle action on $M'$ by, respectively, $\omega'$ and $S'_t$ where now $t$ belongs to the circle $\bbR \mod  \frac{1}{Nq}$.  When dealing with this new circle action, the ``time spent'' in some set will be with respect to the (non normalised) Lebesgue measure on $\bbR \mod  \frac{1}{Nq}$, with total mass $\frac{1}{Nq}$.

The quotient map $\pi: M  \setminus  \mathrm{Fix}(S) \rightarrow M'$ is a covering map, hence an element $\psi \in \Symp_0(M', \omega')$ lifts to a symplectic diffeomorphism $\Psi \in \Symp_0(M, \omega)$ which commutes with $S_{\frac{1}{Nq}}$.  Furthermore,  if $\psi$ is close to the identity then so is $\Psi$.

Recall that one of the steps in the construction of the equidistribution boxes $A_i$ requires a decomposition of the torus $\bbT^n$, into cubes $T$, which is invariant under translation by $S_{\frac{1}{q}}$.   By taking a refinement of that decomposition, which is invariant under translation by $S_{\frac{1}{Nq}}$, we obtain smaller boxes $a_1, \ldots, a_{N'}$\footnote{The $N'$ here is not the same as the $N'$ used in the proof of Lemma \ref{lemm.equidistributionboxes}.} which have the following list of properties:

\begin{enumerate}

\item (Action-Angle Coordinates) Each $a_{j}$ is contained in some $A_i$.  Moreover, $a_j$ is the subset of $A_i$ obtained by restricting the map from the third item in Lemma \ref{lemm.equidistributionboxes}, 
$$
\begin{array}{rcl}
\Xi: P \times T \rightarrow M \\
(s,t) \mapsto \bbT^n_{t}(\sigma(s)),
\end{array}
$$
to $P \times T'$, where $T'$ is one of the cubes arising from the subdivision of the cube $T$  into smaller equal cubes.   The map $\Xi$ defines a symplectomorphism from $\Int(P \times T')$ to $\Int(a_i)$. 

Note that each $A_i$ contains the same number of $a_j$'s.  Therefore, it is sufficient to prove the statement of Lemma \ref{lemm.bigboxes} with the $A_i$'s replaced by $a_j$'s.

 \item Each box  $a_j$ contains at most one point of each orbit under $S_{\frac{1}{Nq}}$,
 so the projection $\pi: M \to M'$ restricts to a symplectomorphism between  $ a_j$ and $\pi( a_j)$, for all $j$. This implies, in particular, that $\Int(\pi(a_j))$ is also symplectomorphic to $\Int(P \times T' )$.
 
\item Note that since $W$ and $\cup_i A_i$ are disjoint and  invariant under the circle action, the sets $\pi(W)$ and  $\cup_j \pi(a_j)$ are disjoint.
\end{enumerate}

We leave it to the reader to check that Lemma \ref{lemm.bigboxes} follows from the  discussion above by setting  $B_i'$'s, $B_i$'s, $\Psi_1$, and $K$ to be, respectively, the lifts of the $b_i'$'s, $b_i$'s,  $\psi_1$, and $\mathcal K$ from the claim below.

\begin{claim} \label{cl:construction_quotient}
There exist two finite collections of polydiscs $b_i' \subset b_i \subset M'$ and $\psi_1 \in \Symp_0(M', \omega')$ such that the following properties are satisfied:

\begin{enumerate}
\item The polydiscs $b_i$ are all symplectomorphic to the standard polydisc $[-r,r]^{2n}$ for some $r>0$, have disjoint interiors, each $b_i$ is included in the interior of some $\pi(a_j)$, and each $\pi(a_j)$ contains the same number of $b_i$'s,

\item $b_i'$ is a sub-polydisc of $b_i$ of the form   $[-r', r']^{2n} \subset [-r, r]^{2n}$, where $r' < r$,

\item $\psi_1$ is $C^\infty$--close to the identity and is compactly supported,

\item Every orbit of the conjugated circle action $\psi_{1} S' \psi_{1}^{-1}$  spends more time than $\frac{1}{Nq}-\frac{\varepsilon'}{Nq}$  in $\mathcal{K} \cup \pi(W)$, where $\mathcal K$ is a compact subset of $\Int\left(\cup_{i} b_{i}'\right)$. 
\end{enumerate}
\end{claim}
  We will prove the above claim in two steps.  We construct the $b_{i}$'s in the first step, and the $b'_{i}$'s in the second step. At each step we have to perturb the action, so that after Step 1 no orbit spends too much time near the boundary of the $\pi(a_j)$'s, and  after Step 2 no orbit spend too much time near the boundary of the $\pi(b_i)$'s.  Figures \ref{fig:transport1} and  \ref{fig:transport2} represent the two steps of our construction.
 
 \bigskip
 
 \noindent {\bf Step 1:}  In the first step of the proof of Claim \ref{cl:construction_quotient}, depicted in Figure \ref{fig:transport1},  we will construct the polydiscs $b_i$ and a symplectomorphism $\theta_1 \in \Symp_0(M', \omega')$ such that the following properties are satisfied:
 \begin{enumerate}
\item The polydiscs $b_i$ are all symplectomorphic to the standard polydisc $[-r,r]^{2n}$ for some $r>0$, have disjoint interiors, each $b_i$ is included in the interior of some $\pi(a_j)$, and each $\pi(a_j)$ contains the same number of $b_i$'s,

\item $\theta_1$ is $C^\infty$--close to the identity and is compactly supported, 

\item Every orbit of the conjugated circle action $\theta_{1} S' \theta_{1}^{-1}$  spends more time than $\frac{1}{Nq} - \frac{\varepsilon'}{2Nq}$  in $ \pi(W) \cup_{i} b_{i}$. 
\end{enumerate}

\begin{figure}[h] 
\centering
\def\svgwidth{1\textwidth}
\input{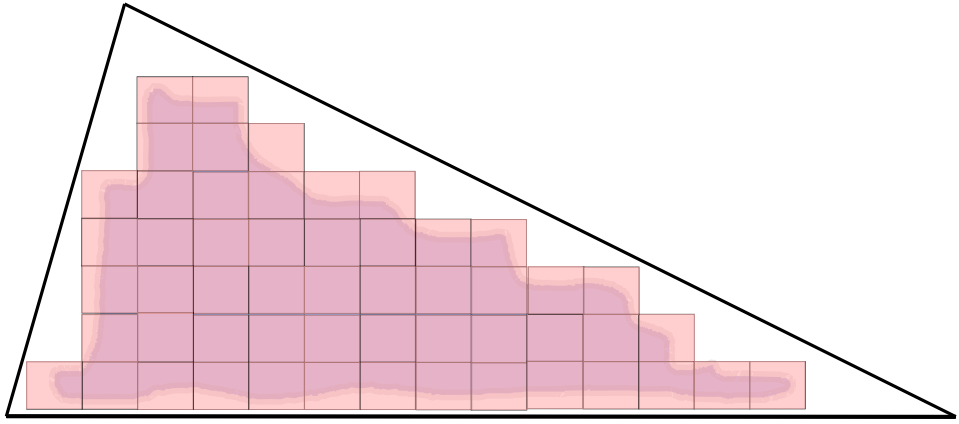}
\caption{\label{fig:transport1} Depiction of Step 1 in the construction of transportation boxes:  The triangle represents $\pi(a_j)$ for a fixed $j$.   The boxes $b_i$ are in pink and the region $\mathcal K_j'$ is shaded purple.  The orbits of $\theta_1 S' \theta_1^{-1}$ spend most of their time in $\pi (W) \cup_j \mathcal{K}_j'$.}
\end{figure}

 Let $\vec V$ be the vector field on $M'$ tangent to the circle   action.  Let $Y$ denote the union of the boundaries of the $\pi(a_{j})$'s. We will now make a perturbation of the vector field so that each of its orbits spends little time near $Y$.

  The boundary of each $\pi(a_{j})$ is a union of submanifolds with boundary.  Let $(Y_{k})_{k =0, \dots , m}$ be a numbering of all these submanifolds, and 
let  $(X_{k})_{k =0, \dots , m}$ denote open hypersurfaces in $M$ such that for each $k$, $X_{k}$ contains $Y_{k}$. 
  Applying our transversality lemma, Lemma \ref{lemma.transversality}, we obtain a $C^\infty$ small and compactly supported symplectorphism $\varphi_0$ of $M'$ such that $X_{0}$ is stably almost transverse to $\varphi_{0 *}\vec V$ on $Y_{0}$.  We repeatedly apply Lemma \ref{lemma.transversality} to obtain $C^\infty$ small symplectomorphisms  $(\varphi_{k})_{k =0, \dots , m}$ of  $M'$ such that for each $k$ and each $j \leq k$, $X_{j}$  is stably almost transverse to $(\varphi_{k} \circ \cdots \circ \varphi_{0})_{*}\vec V$ on  $Y_{j}$. Now, let  $\theta_{1} =  \varphi_{m} \circ \cdots \circ  \varphi_{0}$.  We have that  $X_{k}$ is stably almost transverse to $\theta_{1 *}\vec V$ on $Y_{k}$ for each $0 \leq k \leq m$.  Observe that this implies that  every orbit of the conjugated action $\theta_{1} S' \theta_{1}^{-1}$  meets $Y= \cup Y_{k}$ at most finitely many times.  According to the Thickening Lemma \ref{lemma.thickening}, there exists $\delta>0$ so that every orbit of $\theta_{1} S' \theta_{1}^{-1}$ spends less time than $\frac{\varepsilon'}{2Nq}$ in the $\delta$-neighborhood $O$ of the set $Y$.

 For each $j$, let $\mathcal K_j' = \pi(a_j) \setminus O $; this  is a compact subset of the interior of $\pi(a_j)$.
Let $\mathcal K' = \cup \mathcal K_j'$ and observe that every orbit of  $\theta_{1} S' \theta_{1}^{-1}$ spends more time than $\frac{1}{Nq}- \frac{\varepsilon'}{2Nq}$ in $\mathcal K' \cup \pi(W)$.  
    
To complete the first step of the proof, it remains to show that we can find  symplectomorphic polydiscs $b_i$ with disjoint interiors such that each $b_i$ is contained in some $\pi(a_j)$,  each $\pi(a_j)$ contains the same number of $b_i$'s, and  $\mathcal K' \subset \Int\left(\cup_{i} b_{i}\right)$.
 
  As mentioned above, there exists a symplectic identification of $\Int(\pi(a_j))$ with the interior of a product of the form $P_j \times T$, where $P_j$ is some polytope in $\mathbb R^n$ and $T$ is a cube in the torus $\bb T^n$.     Hence, we may suppose that $\mathcal K_j' \subset \Int(P_j \times T)$.  We will take care of $P_{j}$ and $T$ separately.  Recall that the polytopes $P_j$ were picked to have equal volumes.   

\begin{claim}\label{cl:covering_polytope_by_cubes}
Let $k_j'$ denote the image of $\mathcal K_j'$ under the canonical projection $P_j \times T \rightarrow P_j  \subset \Delta \subset \bbR^n$.  Then,  $k_j'$ may be covered by cubes $e_1, \ldots, e_l$ such that
\begin{enumerate}

\item Each of the cubes $e_1, \ldots, e_l$ is a translation of the cube $[0,\eta]^n$, for some $\eta$,

\item The cubes $e_1, \ldots, e_l$ are all contained in the interior of $P_j$,

 \item The cubes $e_1, \ldots, e_l$ have disjoint interiors,
\item The number of cubes $e_1, \ldots, e_l$  used to cover $k_j'$ does not depend on $j$. 
\end{enumerate}
\end{claim}

To prove the first three items in the above claim one can simply place a grid on $\bb R^n$ whose edge-size is $\eta$, for a sufficiently small $\eta$,  and select the closed cubes from the grid that are included in the interior of $P_{j}$.
To get the last item one must use the fact that the $P_j$'s are polytopes of equal volume in $\bb R^n$.  
Note that if we make the size of the grid converge to zero, then the volume $v_{j}$ of the union of the selected cubes converges to the volume of $P_{j}$. Thus we may choose a common size $\eta$ such that all the $v_{j}$'s are very close to the common volume of the  $P_{j}$'s. If $\eta$ is small enough then many cubes near the boundary $\partial P_{j}$ will actually not intersect $k'_{j}$. Then we may discard some of those unnecessary cubes, in each $P_{j}$, to adjust for the number and get item 4.

Now, we will complete Step 1 of the proof of Claim \ref{cl:construction_quotient}.  All the cubes $T$ in $\bbT^n$
have the same size, so we may identify them with some $T_{0} = [0,\eta]^n$.
We divide $T_{0}$ into very small equal cubes such that the cubes which are contained in the interior of $T$ cover the projection of $\mathcal K_j'$ onto $T$, for every $j$.
The $b_i$'s  are obtained by simply taking the products of these cubes with those from Claim \ref{cl:covering_polytope_by_cubes}. 
As a consequence of Claim \ref{cl:covering_polytope_by_cubes}, each $\pi(a_j)$ contains the same number of the $b_i$'s.  Furthermore, it is not difficult to see that there exists some $r$ such that the $b_i$'s are all symplectomorphic to $[-r,r]^{2n}$.

\medskip

\noindent {\bf Step 2: }  In the second step of the proof of Claim \ref{cl:construction_quotient}, depicted in Figure \ref{fig:transport2},  we construct  polydiscs $b_i'$ and $\theta_2 \in \Symp_0(M', \omega')$ such that 
\begin{enumerate}
\item Each $b_i'$ is a sub-polydisc of $b_i$ of the form   $[-r', r']^{2n} \subset [-r, r]^{2n}$, where $r' < r$,

\item $\theta_2$ is $C^\infty$--close to the identity and is compactly supported,

\item   Every orbit of the conjugated circle action $\psi_1 S \psi_{1}^{-1}$, where $\psi_1 = \theta_2 \circ \theta_1$,  spends more time than $\frac{1}{Nq} - \frac{\varepsilon'}{Nq}$  in $\mathcal{K} \cup \pi(W)$, where $\mathcal K$ is a compact subset of $\Int\left(\cup_{i} b_{i}'\right)$. 
\end{enumerate}
It is clear that the proof of Claim \ref{cl:construction_quotient} will be completed once $b_i'$'s and $\theta_2$, satisfying the above properties, are constructed.  

Recall that every orbit of $ \theta_{1} S' \theta_1^{-1}$ spends more time than $\frac{1}{Nq} -\frac{\varepsilon'}{2Nq}$ in $\pi(W) \cup_i b_i  $.  We can find  a $C^{\infty}$-small symplectomorphism $\theta_2$ such that every orbit of the conjugated action $\theta_2 \theta_{1} S' \theta_1^{-1} \theta_{2}^{-1}$  spends 
\begin{itemize}
\item more time than  $1 -\frac{\varepsilon'}{Nq}$ in $\pi(W) \cup_i b_i  $ and,
\item less time than $\frac{\varepsilon'}{Nq}$ in some small open neighborhood, say  $O'$,  of the union of the boundaries of the polydisc $b_i$.
\end{itemize}
The construction of $\theta_2$ is very similar to that of $\theta_1$ from Step 1 and so it will be omitted.  We let $\psi_1 = \theta_2 \theta_1$.     For each $i$, let $\mathcal K_i =  b_i \setminus O' $; this  is a compact subset of the interior of $b_i$.  One can check that  every orbit of  $\psi_{1} S \psi_{1}^{-1}$ spends more time than $\frac{1}{Nq} - \frac{\varepsilon'}{Nq}$ in $ \pi(W) \cup_i \mathcal K_i  $.   Let $\mathcal{K} = \cup_i \mathcal K_i$.

Recall that each $b_i$ is symplectomorphic to $[-r,r]^{2n}$.  Since the $\mathcal K_i$'s are compact subsets of the $b_i$'s we can find $r'<r$ such that $\mathcal K_i \subset \Int(b_i')$, where $b_i'$ is the sub-polydisc of $b_i$ of the form   $[-r', r']^{2n} \subset [-r, r]^{2n}$.  Hence, we have established that $\mathcal{K}$ is a compact subset of $\Int( \cup_i  b_i')$  and every orbit of  $\psi_{1} S \psi_{1}^{-1}$ spends more time than $\frac{1}{Nq}- \frac{\varepsilon'}{Nq}$ in $ \pi(W) \cup \mathcal K  $.  This completes Step 2, and hence the entirety of the proof of Claim \ref{cl:construction_quotient}.

\begin{figure}[h]
\centering
\def\svgwidth{1\textwidth}
\input{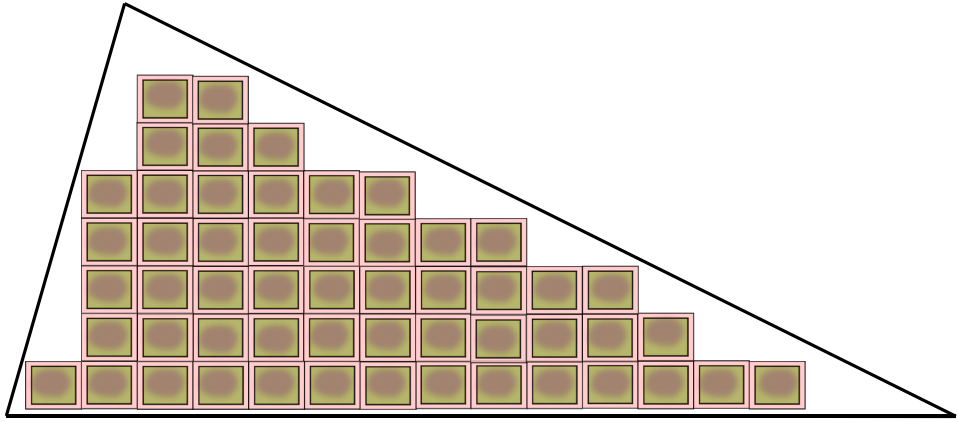}
\caption{\label{fig:transport2} Depiction of Step 2 in the construction of transportation boxes:  The triangle represents $\pi(a_j)$ for a fixed $j$.   The boxes $b_i' \subset b_i$ are in green and pink, receptively, and the regions comprising $\mathcal{K}$ are shaded purple.   The orbits of $\psi_1 S \psi_1^{-1}$ spend most of their time in $\pi (W) \cup \mathcal{K}$.}
\end{figure}
\end{proof}

\subsection{Small boxes}\label{sec:small_boxes}
   We continue to work in the settings of Sections \ref{sec:equidistribution} and \ref{sec:transportation}: We have the equidistribution boxes $A_i$, provided by Lemma \ref{lemm.equidistributionboxes}, the transportation boxes $B_i$, $B_i'$, from Lemma \ref{lemm.bigboxes}, and lastly $\Psi_1 \in \Symp_0(M)$ as described in Lemma \ref{lemm.bigboxes}.

\begin{lemma}[Small Boxes]
\label{lemm.smallboxes}
Let $q \in \cQ$ and $\varepsilon, \varepsilon'>0$ be as in the previous lemmas.
There exist a finite collection of polydiscs $\{c_{k}\}$ in $M$ and $\Psi \in \Symp_0(M, \omega)$ such that the following properties are satisfied:

\begin{enumerate}
\item The $c_k$'s are disjoint, each  $c_k$ is included in some transportation box $B'_{i} = B'(c_k)$, and each transportation box contains the same number of $c_k$'s.  

\item  $S_{\frac{1}{Nq}}$ acts on the $c_k$'s  as a permutation by $Nq$-cycles.

\item $\Psi$ is $C^{\infty}$--close to the identity, its support is disjoint from $\Fix(S)$, and it commutes with $S_{\frac{1}{Nq}}$.

\item Every orbit of the circle action $\Psi S \Psi^{-1}$  spends more time than $1-\varepsilon'$  in $W \cup (\cup_{k} c_{k})$.

\item (Transport) For any small box $c_k$, let $$\mathcal{O}(c_k) = \{S_{\frac{j}{q}} (c_k) \mid j \in \{0, \dots, q-1\} \}$$ be the orbit of $c_k$ under the action of $S_{\frac{1}{q}}$. 
For any two small boxes  $c_{k_1},c_{k_2}$, there exists  $\Phi_{{k_1}{k_2}}$  a compactly supported symplectomorphism of  $M \setminus W$ which commutes with $S_{\frac{1}{q}}$ and has the following properties:
\begin{enumerate}
\item  $\Phi_{{k_1}{k_2}}$ acts as a permutation on the set of all $c_k$'s,
\item  $\Phi_{{k_1}{k_2}} ( \mathcal{O}(c_{k_1}) ) =  \mathcal{O}(c_{{k_2}})$ and $\Phi_{{k_1}{k_2}} ( \mathcal{O}(c_{{k_2}}) ) =  \mathcal{O}(c_{k_1})$,

\item  $\Phi_{{k_1}{k_2}}(c_{k}) \subset B'(c_{k})$ for each small box $c_{k} \notin \mathcal{O}(c_{k_1}) \cup \mathcal{O}(c_{k_2})$.
\end{enumerate}
\end{enumerate}
\end{lemma}

\begin{proof}
\noindent {\bf Preparation for the construction of $c_k$'s:}\\
Recall that $B_i'$ is a sub-polydisc of $B_i$ of the form $[-r', r']^{2n} \subset [-r, r]^{2n}$.  Let $V_i$ be the polydisc in  $B_{i} \setminus \Int(B'_{i})$ of the form $[r', r' + 2\eta] \times [-\eta, \eta]^{2n-1}  $, where $2\eta < r- r'.$ 
Note that the collection of $V_i$'s is invariant under the action of $S_{\frac{1}{Nq}}$.


\begin{figure}[h] 
\centering
\def\svgwidth{.6\textwidth}
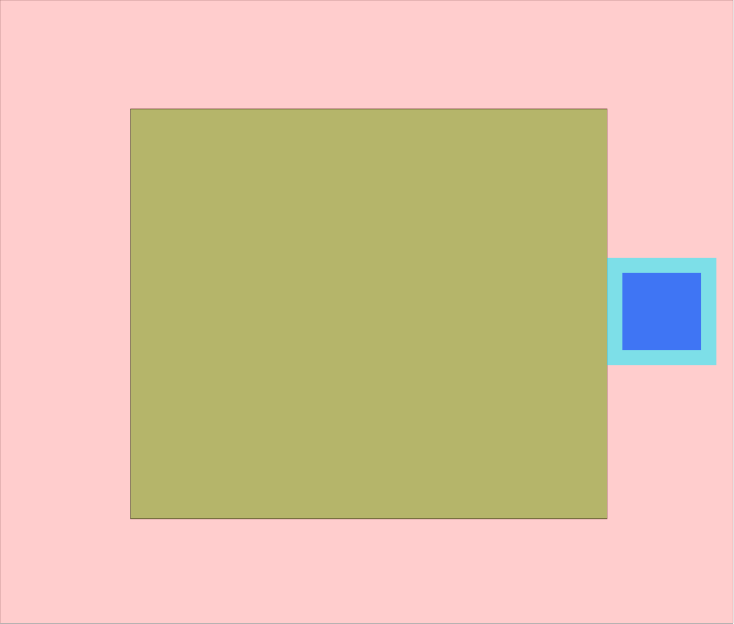
\caption{\label{fig:Small1} As before, $B_i' \subset B_i$ are in pink and green, respectively.  The set $V_i$ is in light blue and $U_i(\delta)$ in dark blue.}
\end{figure}

For each $0< \delta < \eta $, let $U_i(\delta)$ be the sub-polydisc in $V_i$ corresponding to $[r'+\eta -\delta, r' + \eta + \delta] \times [-\delta, \delta]^{2n-1}$;  see Figure \ref{fig:Small1}.  Note that the collection of $U_i(\delta)$'s is also invariant under the action of $S_{\frac{1}{Nq}}$.


\begin{claim}\label{cl:transport_small_boxes}
There exists $\delta_0 >0$ with the following property: Take any $\delta \leq \delta_0$ and consider any two  $V_{i_1}, V_{i_2}$ which have disjoint orbits under the action of $S_{\frac{1}{q}}$, i.e.\ $S_{\frac{j}{q}} (V_{i_1}) \neq V_{i_2}$ for any $j \in \{0, \ldots, q-1 \} $.  Then, there exists a compactly supported symplectomorphism $\Upsilon_{i_1i_2} \in  \Symp_0( M \setminus (W \cup (\cup_i B'_{i} )))$  such that
\begin{enumerate}
\item $\Upsilon_{i_1i_2}(U_{i_1}(\delta)) = U_{i_2}(\delta)$ and $\Upsilon_{i_1i_2}(U_{i_2}(\delta)) = U_{i_1}(\delta)$, 
\item $\Upsilon_{i_1i_2} S_{\frac{1}{q}} = S_{\frac{1}{q}}\Upsilon_{i_1i_2}$.
\end{enumerate}  
\end{claim}
\begin{proof}[Proof of Claim \ref{cl:transport_small_boxes}]

Observe that $S_{\frac{1}{q}}$ acts on $M \setminus (W  \cup (\cup_i B'_{i}) )$; this is because $W$ is invariant under the circle action and $\cup_i B_i'$'s is invariant under the action of $S_{\frac{1}{Nq}}$.  Furthermore, this action is free.  Therefore, we may consider the quotient symplectic manifold $M'' = M \setminus (W \cup (\cup_i B'_{i}) )  / S_{\frac{1}{q}}$.      Let $\pi: M \setminus (W \cup (\cup_i B'_{i}) ) \rightarrow M''$ denote the quotient map.  

We leave it to the reader to check that proving Claim \ref{cl:transport_small_boxes} may be reduced to proving the following statement on $M''$.  There exists $\delta_0 > 0$ with the property that for every $\delta < \delta_0$ and any two $i_1, i_2$,  such that  $\pi(\Int (V_{i_1})) \neq \pi(\Int (V_{i_2}))$, we can find a compactly supported symplectomorphism $\Upsilon_{i_1i_2}$ of $M''$ such that 
$\Upsilon_{i_1i_2} (\pi(U_{i_1}(\delta))) = \pi(U_{i_2}(\delta))$ and $\Upsilon_{i_1i_2}(\pi(U_{i_2}(\delta))) = \pi(U_{i_1}(\delta))$.   

 The fact $\Upsilon_{i_1i_2}$ of the previous paragraph exists for small enough values of $\delta$ is a consequence of Lemma \ref{lemm:transport_lemma} applied in the symplectic manifold $M''$.
 
 The map $\Upsilon_{i_1i_2}$ is depicted by the dotted line in Figure \ref{fig:Small3}.
\end{proof}

\noindent {\bf Construction of $c_k$'s:}
Recall that $B_i'$ is symplectomorphic to $[-r', r']^{2n}$.   We subdivide each $B_i'$ into polydiscs all of which are symplectomorphic to $[-\delta', \delta']^{2n}$, for some $\delta'$ smaller than the $\delta_0$ given by the previous claim; we will denote the collection of these polydiscs by $\{c_k'\}$.
Each $B_i'$ contains the same number of $c_k'$'s and the $c_k'$'s have disjoint interiors.  Since the collection of $B_i'$'s is invariant under the action of $S_{\frac{1}{Nq}}$, we can ensure that the collection of $c_k'$'s is also invariant under the action of $S_{\frac{1}{Nq}}$.

Recall from Lemma \ref{lemm.bigboxes} that  every orbit of the conjugated action $\Psi_1 S \Psi_1^{-1}$ spends more time than $1 - \varepsilon'$ inside $K \cup W$, where $K$ is a compact subset of $\Int\left(\cup_{i} B_{i}'\right)$.  As was done in the proof of of Lemma \ref{lemm.bigboxes}, using Transversality Lemma \ref{lemma.transversality} and Thickening Lemma \ref{lemma.thickening}, we can find $\Psi_2$ such that every orbit of the conjugated action $\Psi_2 \Psi_1 S \Psi_1^{-1} \Psi_2^{-1} $ spends more time than $1 - \varepsilon'$ in $W \cup_k c_k$ where $c_k$ is the sub-polydisc of $c_k'$ of the form $[-\delta,  \delta]^{2n} \subset [-\delta', \delta']^{2n}$, with $ \delta$ being slightly smaller than $\delta'$.  Furthermore, the map $\Psi_2$ can be picked such that it commutes with $S_{\frac{1}{Nq}}$. We will not describe the construction of $\Psi_2$ as it is very similar to that of the maps $\Psi_1$ and $\psi_1$ from Lemma \ref{lemm.bigboxes}; see in particular the first paragraph of the proof of Claim \ref{cl:construction_quotient}.  The symplectomorphism $\Psi$ is the composition $\Psi_2 \circ \Psi_1$. Note that $\Psi$ and the $c_{k}$'s satisfy items 1--4 of the lemma. It remains to construct the maps $\Phi_{{k_1}{k_2}}$ required by the last item.  See Figure \ref{fig:Small2} for a depiction of the small boxes $c_k$.

\begin{figure}[h] 
\centering
\def\svgwidth{.8\textwidth}
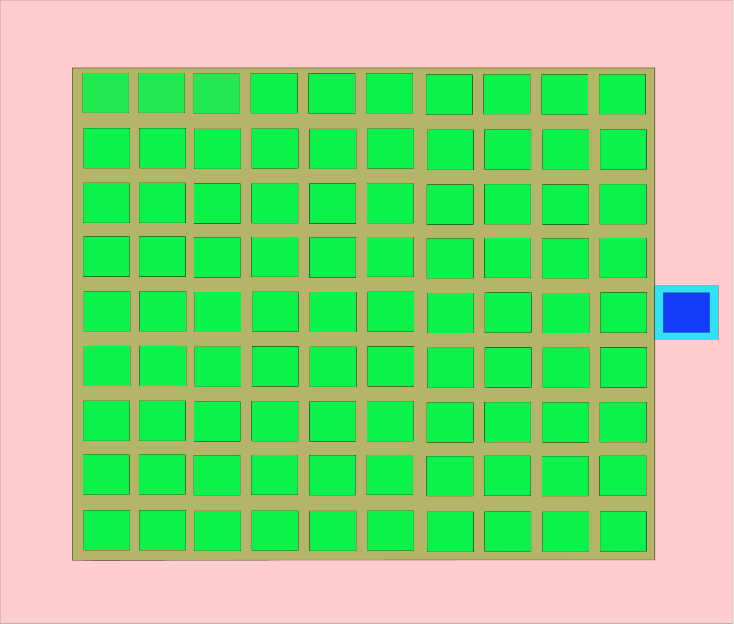
\caption{\label{fig:Small2} The small boxes $c_k$, in bright green, are added to the previous picture;  they are symplectomorphic to $U(c_k)$ depicted in dark blue.  Every orbit of $\Psi S \Psi^{-1}$ spends more time than $ 1- \varepsilon'$ in $W \cup_k c_k$. \newline
    Any two small boxes $c_{k_1}, c_{k_2}$, in bright green, can be exchanged via swapping several adjacent boxes.  Similarly, $U(c_k)$ can be swapped with the small box adjacent to it.}
\end{figure}

\medskip

\noindent {\bf Proof of the last item (Transport): }

 We will need to introduce a bit of notation for the remainder of the proof.  For any small box $c_k$, there exists (unique) $i_k$ such that $ c_k \subset B'_{i_k} \subset B_{i_k}$. We will denote $B'_{i_k}, B_{i_k}$ by $B'(c_k), B(c_k)$, respectively.    Similarly, we  will denote $U_{i_k}(\delta)$ and $V_{i_k}$ by $U(c_k)$ and $V(c_k)$, respectively.   Observe that $U(c_k)$ and $c_k$ are both polydiscs symplectomorphic to $[-\delta, \delta]^{2n}.$   

    Roughly speaking, our goal here is to find a symplectomorphism $\Phi_{k_1 k_2}$ which exchanges  two given small boxes $c_{k_1}$ and $c_{k_2}$  while leaving the remaining small boxes more or less untouched.  Of course, the remaining boxes cannot be left entirely untouched because  $\Phi_{k_1 k_2}$ must commute with $S_{\frac{1}{q}}$ and so we are automatically forced to swap $\mathcal{O}(c_{k_1})$ and $ \mathcal{O}(c_{{k_2}})$.  However, the exchange may  be achieved such that the remaining $c_k$'s do not leave their transportation boxes.  The rough idea of the construction is as follows: First, any two small boxes $c_{k_1}, c_{k_2}$ which are in the same transportation box $B'(c_k)$ can be swapped; the remaining small boxes in $B'(c_k)$ might be affected, however, they will remain in $B'(c_k)$; this is the content of Claim \ref{cl:swapping_small_boxes} and its proof is rather evident in dimension two, see Figure \ref{fig:Small3}, and the  argument generalizes to higher dimensions.     Second, a similar reasoning can be used to obtain a symplectomorphism which swaps $U(c_k)$ with the small box adjacent to it; see Figure \ref{fig:Small3} and Claim \ref{cl:transport_to_exit_box}.  Third, we can exchange $U(c_{k_1})$ and $U(c_{k_2})$;  see Figure \ref{fig:Small3} and Claim \ref{cl:transport_small_boxes}. Finally, we obtain $\Phi_{k_1 k_2}$ by combining the above facts.

 \begin{figure}[h] 
\centering
\def\svgwidth{.8\textwidth}
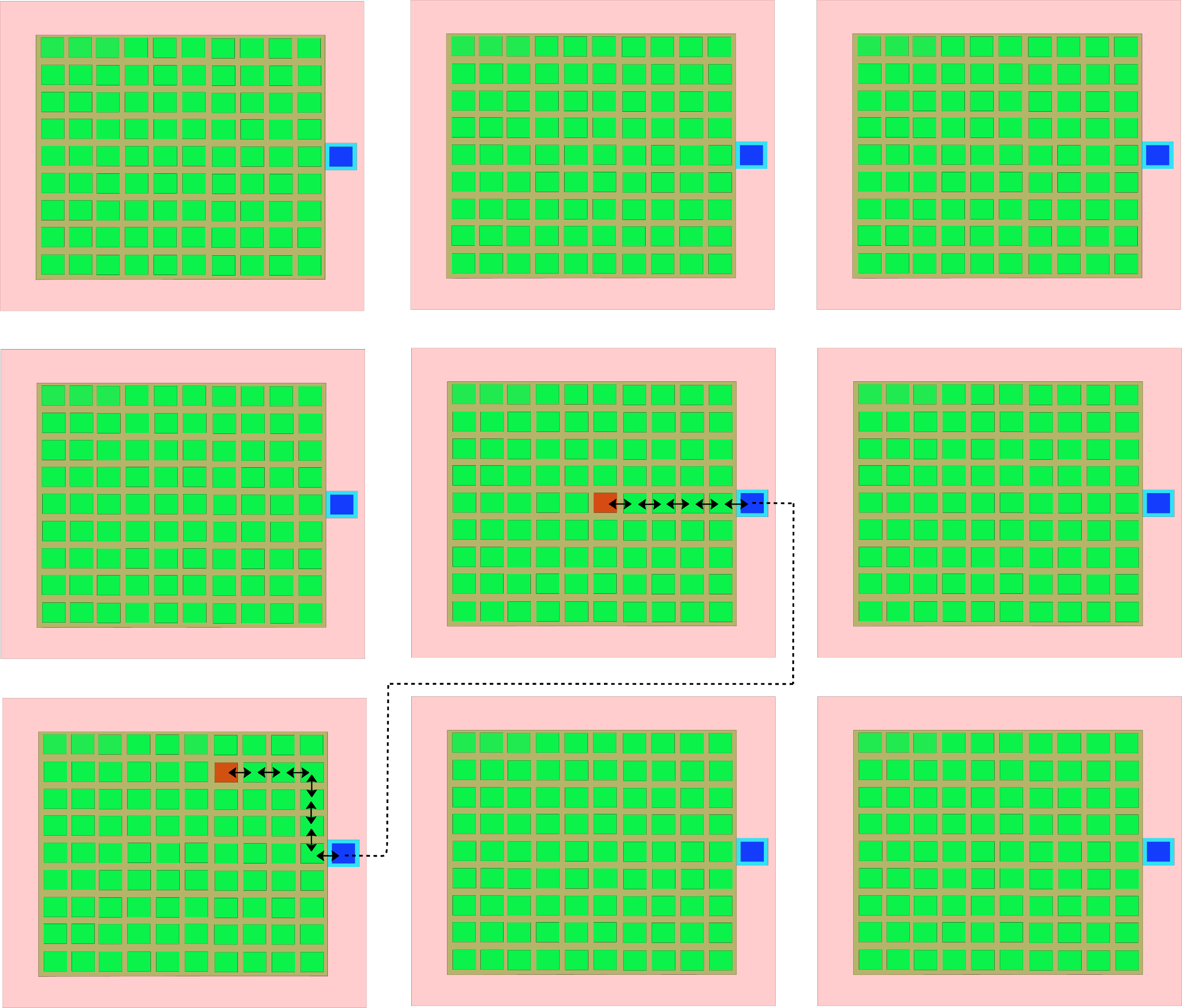
\caption{\label{fig:Small3} 
Depiction of $\Phi_{k_1k_2}$ exchanging $c_{k_1}$ and $c_{k_2}$ which are colored in red.  First, via swapping adjacent small boxes we map $c_{k_1}$ and $c_{k_2}$ to $U(c_{k_1})$ and $U(c_{k_2})$, respectively.  The swaps are depicted by the two-headed arrows.  We then exchange $U(c_{k_1})$ and $U(c_{k_2})$ using the space in between the transportation boxes; the dotted line depicts a path taken by a symplectic isotopy swapping $U(c_{k_1})$ and $U(c_{k_2})$.  The existence of such isotopy is guaranteed by Claim \ref{cl:transport_small_boxes}.
 }
\end{figure}

  We will be needing the following two claims.

  \begin{claim}   \label{cl:swapping_small_boxes}
 Fix a small box $c_k$ and  let $c_{k_1}, c_{k_2}$ be any two small boxes which are contained in $B(c_{k})$.  Then, there exists $\Phi_{{k_1}{k_2}} \in \Symp_0(M, \omega)$ with the following property:
    \begin{enumerate}
    \item $\Phi_{{k_1}{k_2}} $ is supported in $ \cup_j S_{\frac{j}{q}} (B'(c_k)) \subset \cup_j S_{\frac{j}{q}} (B(c_k)) \subset M \setminus W$, 
    \item $\Phi_{{k_1}{k_2}} $ commutes with $S_{\frac{1}{q}}$,
    \item $\Phi_{{k_1}{k_2}} (c_{k_1}) = c_{k_2}$ and  $\Phi_{{k_1}{k_2}} (c_{k_2}) = c_{k_1}$,
    \item $\Phi_{{k_1}{k_2}} $ acts as a permutation on the set $\{c_{k'}:  c_{k'} \subset B(c_k)\}$ .  
    \end{enumerate}
\end{claim}
\begin{proof}[Proof of Claim \ref{cl:swapping_small_boxes}]
Since the set of $B_i'$'s is invariant under the action of $S_{\frac{1}{q}}$, it is enough to prove the following:   There exists a symplectomorphism $\phi_{{k_1}{k_2}} $  which is compactly supported in $B'(c_k)$, is isotopic to the identity,  and has the following properties:
    \begin{enumerate}
   \item  $\phi_{{k_1}{k_2}} (c_{k_1}) = c_{k_2}$ and  $\phi_{{k_1}{k_2}} (c_{k_2}) = c_{k_1}$,
   \item  $\phi_{{k_1}{k_2}} $ acts as a permutation on the set $\{c_{k'}:  c_{k'} \subset B(c_k)\}$. 
    \end{enumerate}
  We leave it to the reader to check that the existence of  $\phi_{{k_1}{k_2}} $ can be deduced from Lemma \ref{lemm:polydisc_swapping_lemma}.  The map $\phi_{{k_1}{k_2}} $ corresponds to a symplectomorphism which, within the same transportation box $B(c_k)$,  exchanges two bright green squares in either of Figures \ref{fig:Small2} or \ref{fig:Small3}.  This can be achieved via a composition of symplectomorphisms which swap adjacent squares as depicted in Figure \ref{fig:Small3}.
\end{proof}

\begin{claim}   \label{cl:transport_to_exit_box}
  For each small box $c_k$, there exists  $\Theta_k \in \Symp_0(M, \omega)$  with the following properties:
    \begin{enumerate}
    \item $\Theta_k$ is supported in $ \cup_j S_{\frac{j}{q}} (B(c_k)) \subset M \setminus W$, 
    \item $\Theta_k$ commutes with $S_{\frac{1}{q}}$,
    \item $\Theta_k (c_k) = U(c_k)$, 
    \item  $\Theta_k$ acts as a permutation on the collection of polydiscs $\{U(c_k)\} \cup \{c_{k'}: c_{k'}  \subset B(c_k)\}$.  
    \end{enumerate}
\end{claim}
\begin{proof}[Proof of Claim \ref{cl:transport_to_exit_box}]
   Since the set of $B_i$'s is invariant under the action of $S_{\frac{1}{q}}$, it is enough to prove the following:   For each small box $c_k$, there exists a symplectomorphism $\theta_k$ which is compactly supported in $B(c_k)$, is isotopic to the identity, with the following properties:
    \begin{enumerate}
    \item $\theta_k (c_k) = U(c_k)$,
    \item  $\theta_k$ acts as a permutation on the collection of polydiscs $\{U(c_k)\} \cup \{c_{k'}: c_{k'}  \subset B(c_k)\}$.  
    \end{enumerate}
 We leave it to the reader to check that the existence of  $\theta_k $ can be deduced from Lemma \ref{lemm:polydisc_swapping_lemma}.   In Figure \ref{fig:Small3},   $\theta_k $ corresponds to a symplectomorphism which exchanges the dark blue square with the red square in the adjacent transportation box; it is obtained as a composition of maps that permute adjacent squares, as indicated by the two-headed arrows.
\end{proof}
We now prove the last item (Transport) in the statement of Lemma \ref{lemm.smallboxes} using Claims \ref{cl:transport_small_boxes}, \ref{cl:swapping_small_boxes},  \ref{cl:transport_to_exit_box}.   

Note that it is sufficient to prove the statement upto replacing $c_{k_1}$ or $c_{k_2}$ with any element of  $\mathcal{O}(c_{k_1})$ or $\mathcal{O}(c_{k_2})$, respectively.  First, we consider the simpler case where there exists $j$ such that $S_{\frac{j}{q}}(c_{k_2}) \subset B(c_{k_1})$. Then, upto replacing $c_{k_2}$ with $S_{\frac{j}{q}}(c_{k_2}) $, we may assume that  $B(c_{k_1}) = B(c_{k_2})$.  In this case, $\Phi_{{k_1}{k_2}}$ is given by Claim \ref{cl:swapping_small_boxes}. 

 Next, we treat the case where  $B(c_{k_1}) \neq B(c_{k_2})$ even upto  replacing $c_{k_1}, c_{k_2}$ with elements of  $\mathcal{O}(c_{k_1}), \mathcal{O}(c_{k_2})$; Figure \ref{fig:Small3} depicts $\Phi_{{k_1}{k_2}}$ in this scenario.  Let $\Theta_{k_1}, \Theta_{k_2}$ be as given by Claim \ref{cl:transport_to_exit_box}; note that these two maps have disjoint supports.   Then,  define

  $$ \Phi_{{k_1}{k_2}} = \Theta_{k_1}^{-1} \Theta_{k_2}^{-1} \Upsilon \Theta_{k_2} \Theta_{k_1},$$
  where $\Upsilon \in \Symp_0(M\setminus W) $ commutes with $S_{\frac{1}{q}}$   and satisfies  $\Upsilon(U(c_{k_1} ) )= U(c_{k_2})$ and $\Upsilon(U(c_{k_2} )) = U(c_{k_1})$; the existence of $\Upsilon$ is guaranteed by Claim \ref{cl:transport_small_boxes} and is depicted by the dotted line  in Figure \ref{fig:Small3}.  We leave it to the reader to check that $ \Phi_{{k_1}{k_2}}$ satisfies all the requirements of the last item of Lemma \ref{lemm.smallboxes}.  
\end{proof}

\subsection{From boxes to Proposition~\ref{prop.main}}
Having proven the lemmas of Sections \ref{sec:equidistribution}, \ref{sec:transportation}, and \ref{sec:small_boxes}, we are now well-positioned to prove the following proposition which in turn will entail Proposition~\ref{prop.main}.

\begin{prop}
\label{prop.main2}
For any $q \in \cQ$ and any $\varepsilon, \varepsilon' > 0$,  there exist $h \in \Symp_0(M, \omega)$ and  $A_1, \ldots, A_N$ closed subsets of $M$ such that:

\begin{enumerate}
\item The sets $A_{i}$ satisfy the Equidistribution Property from Lemma \ref{lemm.equidistributionboxes}, 

\item The support of $h$ is disjoint from $\Fix(S)$ and $h S_{\frac{1}{q}} = S_{\frac{1}{q}} h $,

\item Every orbit of the conjugated action $h S h^{-1}$ is almost equidistributed among the sets $A_i$ in the following sense:  There exists $ E \subset M$ such that for each $x \in M$ we have  $\Leb_{\bbS^1}(\{ t \in \bbS^1:  h S_{t} h^{-1}(x) \in E\}) < \varepsilon'$ and the following properties are satisfied:
\begin{enumerate}
\item   $ E \subset \cup_i A_i $ and $ \partial A_i \subset E$ for every $i$,
\item For each $x \in M$, let $I_i(x): = \{t\in \bbS^1: h S_{t} h^{-1}(x) \in A_i \setminus E\}$,  Then, $\mathrm{Leb}(I_i(x)) = \mathrm{Leb}(I_j(x))$ for all $i,j$.
\end{enumerate}
\end{enumerate}
 \end{prop}

 \begin{proof}[Proof of Proposition \ref{prop.main2}]
We will be applying the lemmas of the previous sections with the given $q$, $\varepsilon$ and $\varepsilon'$.
 Lemma \ref{lemm.equidistributionboxes} gives us the equidistribution boxes $A_1, \ldots, A_N$.

 Applying  Lemma \ref{lemm.smallboxes} we obtain small boxes $\{c_k\}$ and $\Psi \in \Symp_0(M, \omega)$ satisfying, among others, the following properties:  

\begin{enumerate}
\item Each $A_i$ contains the same number of small boxes $c_k$,

\item  $S_{\frac{1}{Nq}}$ acts by a cyclic permutation of order $Nq$ on the $c_k$'s. 

\item $\Psi$ is $C^{\infty}$--close to the identity, its support is disjoint from $\Fix(S)$, and it commutes with $S_{\frac{1}{Nq}}$,

\item Every orbit of the circle action $\Psi S \Psi^{-1}$  spends more time than $1-\varepsilon'$  in $ W \cup_{k} c_{k} $. 
\end{enumerate}

   Given a small box $c$, we will denote by 
   $$
   \mathcal{O}_{q}(c) = \{S_{\frac{j}{q}} (c) \mid j  = 0, \dots, q-1 \}, \ \ \ \ 
   \mathcal{O}_{Nq}(c) = \{S_{\frac{j}{Nq}} (c) \mid j = 0, \dots, Nq-1 \} 
   $$
the orbits of $c$ respectively under the actions of $S_{\frac{1}{q}}$  and $S_{\frac{1}{Nq}}$.
 Using Lemma \ref{lemm.smallboxes}, we can prove the following claim.
\begin{claim}\label{cl:distributing_small_boxes}
There exists $\Theta  \in \Symp_0(M)$ which is compactly supported in $ M \setminus W$ such that 

\begin{enumerate}
\item $\Theta S_{\frac{1}{q}} = S_{\frac{1}{q}} \Theta $,

\item   For any small box $c$, the interior of each equidistribution box $A$ contains exactly $q$ of the elements of the set $$ \Theta ( \mathcal{O}_{Nq}(c)) = \{ \Theta(c), \Theta( S_{\frac{1}{Nq}}(c) ) , \dots, \Theta( S_{\frac{Nq-1}{Nq}}(c)) \}.$$
\end{enumerate}
\end{claim}
\begin{proof}[Proof of Claim \ref{cl:distributing_small_boxes}]
   We begin by explaining the main  idea  of the proof of this before proceeding to the give the details of the proof.

  Given a small box $c$, and any symplectomorphism $\Theta$, we will say $\Theta (\mathcal{O}_{Nq}(c))$ is equidistributed if each equidistribution box $A$ contains exactly $q$ of its  elements.   Note that if  $\mathcal{O}_{Nq}(c))$ is equidistributed for every $c$, then we are done with $\Theta = \Id$.  If not we can find transportation boxes say $A_1, A_2$, and a small box $c_{k_1}$ contained in $A_1$ such that $A_1$ contains more than $q$ of the elements of $ \mathcal{O}_{Nq}(c_{k_1})$ and $A_2$ contains less than $q$ of them.  Since each $A_i$ contains exactly the same number of small boxes, we can find a small box, which we denote by $c_{k_2}$, such that $A_2$ contains more than $q$ of the elements of $\mathcal{O}_{Nq}(c_{k_2})$.    By the transport item of Lemma \ref{lemm.smallboxes},  there exists  $\Phi_{{k_1}{k_2}}$,  a compactly supported symplectomorphism of  $M \setminus W$, which commutes with $S_{\frac{1}{q}}$ and has the following property:   it exchanges the orbits $\mathcal{O}_q(c_{k_1})$ and $\mathcal{O}_q(c_{k_2})$. 
   As for the other small boxes, $\Phi_{{k_1}{k_2}}$ leaves them nearly unchanged in the sense that  $c$ and $\Phi_{{k_1}{k_2}}(c)$ remain in the same equidistribution box.  We see that after applying  $\Phi_{{k_1}{k_2}}$, $A_1$ will contain one less of the elements of $ \mathcal{O}_{Nq}(c_{k_1})$ and $A_2$ will contain one more.  Repeating this process will allow us to construct the map $\Theta$ as the compositions of all such $\Phi_{{k_1}{k_2}}$'s.

  We will now proceed to give more details of the proof.  As will be explained below, we will successively construct, for $k=1, \dots$,  symplectomorphisms  $\Theta_k$ which are compactly supported in $M\setminus W$, commute with $S_{\frac{1}{q}}$,  act as a permutation on the collection of small boxes and such that $\Theta_{k} \circ \cdots \circ \Theta_{1} (\mathcal{O}_{Nq}(c_k))$ is equidistributed.
 Furthermore, each $\Theta_k$ will have the following additional property: For each small box $c$ denote by $A(c)$ the equidistribution box which contains it.  Then, for all $1 \leq i \leq k-1$ and every $c\in \mathcal{O}_{Nq}(c_i)$ we have $$A ( \Theta_k \circ   \Theta_i \cdots \circ \Theta_1 (c)) = A ( \Theta_i \cdots \circ \Theta_1(c)).$$
  This implies, in particular,  that $\Theta_k \circ \cdots \circ \Theta_1 (\mathcal{O}_{Nq}(c_i))$ is equidistributed for all $1 \leq i \leq k$.   Once $\Theta_k$ with such properties is constructed we can simply set $\Theta$ to be the composition of all the $\Theta_k$'s.
  
   Leaving the case where $k=1$ to the reader, we will now describe the construction of $\Theta_k$, assuming $\Theta_1, \ldots, \Theta_{k-1}$ have been constructed.     If $ \Theta_{k-1} \circ \cdots \circ \Theta_1 (\mathcal{O}_{Nq}(c_k))$ is equidistributed we set $\Theta_k = \Id$.  If not, we can find two equidistribution boxes say $A_1, A_2$, such that $A_1$ contains more than $q$ of the elements of $ \Theta_{k-1} \circ \cdots \circ \Theta_1 (\mathcal{O}_{Nq}(c_k))$ and $A_2$ contains less than $q$ of them.  
By induction we know that $A_{2}$ contains exactly $q$  of the elements of $ \Theta_{k-1} \circ \cdots \circ \Theta_1 (\mathcal{O}_{Nq}(c_{k'}))$ for $k' < k$,  
   and moreover each $A_i$ contains exactly the same number of small boxes. Thus, there exists $k'>k$ such that  $A_2$ contains more than $q$ of the elements of $\Theta_{k-1} \circ \cdots \circ \Theta_1 (\mathcal{O}_{Nq}(c_{k'}))$ (in the sequel we will just need one of these elements).

   
   Let  $c_{k_1}, c_{k_2}$ denote $\Theta_{k-1} \circ \cdots \circ \Theta_1 (c_k) $ and $\Theta_{k-1} \circ \cdots \circ \Theta_1 (c_{k'})$, respectively.   By the transport item of Lemma~\ref{lemm.smallboxes},  there exists  $\Phi_{{k_1}{k_2}}$  a compactly supported symplectomorphism of  $M \setminus W$ which commutes with $S_{\frac{1}{q}}$ and
satisfies properties (a), (b), (c) of  Lemma~\ref{lemm.smallboxes}.

We leave it to the reader to check that property (c) has the following consequence: for all $1 \leq i \leq k-1$ and every $c\in \mathcal{O}_{Nq}(c_i)$ we have $$A (  \Phi_{{k_1}{k_2}}  \circ   \Theta_i \cdots \circ \Theta_1 (c)) = A ( \Theta_i \circ \cdots \circ \Theta_1(c)).$$
  This implies, in particular,  that $ \Phi_{{k_1}{k_2}} \circ \cdots \circ \Theta_1 (\mathcal{O}_{Nq}(c_i))$ is equidistributed for all $1 \leq i \leq k-1$. 

By property (b),  the number of elements of  $ \Phi_{{k_1}{k_2}} \circ \Theta_{k-1} \circ \cdots \circ \Theta_1 (\mathcal{O}_{Nq}(c))$ which are contained in $A_1$ is one less than the number of elements of $\Theta_{k-1} \circ \cdots \circ \Theta_1 (\mathcal{O}_{Nq}(c))$  which are contained in $A_1$.  It follows that by repeatedly applying the transport item of Lemma~\ref{lemm.smallboxes}, we can continue the above process to obtain other $\Phi_{{k_i}{k_j}}$'s the composition of all of which gives the map $\Phi_k$. 
\end{proof}

We will now show that Claim \ref{cl:distributing_small_boxes} implies Proposition \ref{prop.main2}.  Indeed, let $h= \Theta  \Psi$ and consider the conjugated circle action $h S h^{-1}$.  It is clear that the first two items in the statement of the proposition hold.  We must prove the third item. We define the set $E:= \cup A_i \setminus \Theta  ( \cup_{k} c_{k} )$.  It is clear that $E \subset  \cup A_i $  and $\partial A_i \subset E$ for each $i$. 

Observe that, by point 4 of Lemma \ref{lemm.smallboxes}, every orbit of the circle action $h S h^{-1}$ spends more time than $1- \varepsilon'$ in  $\Theta ( W \cup_{k} c_{k}  ) =  W \cup \, \Theta  ( \cup_{k} c_{k} ) $. Hence,  we immediately obtain $\Leb_{\bbS^1}(\{ t \in \bbS^1:  h S_{t} h^{-1}(x) \in E\}) < \varepsilon'$ for every $x$.

It remains to show that $\mathrm{Leb}(I_i(x)) = \mathrm{Leb}(I_j(x))$ for all $i,j$.  This is equivalent to showing that the quantity $$\mathrm{Leb} (\{t\in \bbS^1 : h S_{t} h^{-1}(x) \in A_i \cap \Theta  ( \cup_{k} c_{k} ) \} )$$ does not depend on $i$.  Now, using the action of $S_{\frac{1}{Nq}}$ on the $c_{k}$'s, we see that this quantity equals
$$\sum_{\mathcal{O}_{Nq} (c)  } \mathrm{Leb} (\{t\in \bbS^1 : hS_{t} h^{-1}(x) \in A_i \cap \Theta  \left( \mathcal{O}_{Nq} (c) \right)  \} ),$$
where the sum is taken over distinct $\mathcal{O}_{Nq}(c)$'s.  Hence,  it is sufficient to show that  for any small box $c$ the quantity
$$\mathrm{Leb} ( \; \{t\in \bbS^1 : h S_{t} h^{-1}(x) \in A_i \cap \Theta  \left(  \mathcal{O}_{Nq}(c) \right) \} \; )$$ deos not depend on $i$.   By Claim \ref{cl:distributing_small_boxes}, there exists $q$ elements, say $c_1, \ldots, c_q \in \mathcal{O}_{Nq}(c)$ such that $A_i \cap \Theta  \left(  \mathcal{O}_{Nq}(c) \right) = \Theta (c_1)  \cup \ldots \cup \Theta (c_q)$.  Thus, 

$$\mathrm{Leb} ( \{t\in \bbS^1 : h S_{t} h{-1}(x) \in A_i \cap \Theta  \left(  \mathcal{O}_{Nq}(c) \right) \} )$$ 
$$ =  \sum_{j=1}^{q} \mathrm{Leb} (  \{t\in \bbS^1 : hS_{t} h^{-1}(x) \in \Theta  \left(  c_j \right) \} ).  $$
Now, recall that $h=  \Theta \Psi$ and so  $\mathrm{Leb} (  \{t\in \bbS^1 : hS_{t} h^{-1}(x) \in \Theta  \left(  c_j \right) \} )$  coincides with  $\mathrm{Leb} (  \{t\in \bbS^1 : \Psi S_{t} \Psi^{-1}(z) \in   c_j \} )$, where $z = \Theta^{-1}(x)$.  Lastly, because $\Psi$ commutes with $S_{\frac{1}{Nq}}$ and $c_j \in \mathcal{O}_{Nq}(c)$,  we have that $\mathrm{Leb} (  \{t\in \bbS^1 : \Psi S_{t} \Psi^{-1}(z) \in   c_j \} ) = \mathrm{Leb} (  \{t\in \bbS^1 : \Psi S_{t} \Psi^{-1}(z) \in   c \} ) $.  Hence,

 $$\mathrm{Leb} ( \{t\in \bbS^1 : hS_{t} h^{-1}(x) \in A_i \cap \Theta  \left(  \mathcal{O}_{Nq}(c) \right) \} )$$ 
$$ =  q \;  \mathrm{Leb} (  \{t\in \bbS^1 : \Psi S_{t} \Psi^{-1}(z) \in   c \} ), $$
which clearly does not depend on $i$; the above equality follows from the second item of Claim \ref{cl:distributing_small_boxes}.  This finishes the proof of Proposition \ref{prop.main2}.
\end{proof}

It remains to explain why Proposition \ref{prop.main} follows from Proposition \ref{prop.main2}.

\begin{proof}[Proof of Proposition \ref{prop.main}]
 Let $\cU$ denote an open neighborhood of $\Conv(\cE)$.  Clearly, the symplectomorphism $h$, given to us by Proposition \ref{prop.main2}, satisfies the first two items of Proposition \ref{prop.main}.  It remains to prove the third item, that is, if $\varepsilon$ and $\varepsilon'$ are small enough, for every $x \in M$, the push-forward of  $\mathrm{Leb}_{\bbS^1}$, the Lebesgue measure on the circle, under the map $t \mapsto h S_{t} h^{-1}(x)$, belongs to $\cU$.

Fix $x \in M$ and let $\mu$ be the push-forward of  $\mathrm{Leb}_{\bbS^1}$ under the map $t \mapsto h S_{t} h^{-1}(x)$.  Fix $\varepsilon>0$ and let $\varepsilon' = \frac{\varepsilon}{2N}$. We leave it to the reader to check that, as a consequence of the third item of Proposition \ref{prop.main2}, $\mu$ has the property that
\begin{equation}\label{eq:estimate_mu}
\sum_{i=1}^N \vert \mu(A_i) - \frac{\alpha}{N} \vert < (N + 1) \varepsilon' < \varepsilon,
\end{equation}
where $\alpha =  \mu( \cup_i A_i)$.
We will show  that any probability measure which satisfies  the above property for sufficiently small $\varepsilon> 0$ belongs to $\cU$.   

Let $\nu \in \cP(M)$ and recall that a basis of open neighborhoods of $\nu$ for the week topology is given by the collection of sets of the the form 
$$U_{\delta, f_1, \ldots, f_k}(\nu):= \left\{\beta\in \cP(M):  \, \left\vert \int f \, d\nu - \int f \, d\beta \, \right\vert < \delta, \;  \forall f \in \{f_1, \ldots, f_k \}  \right\},$$
where $\delta > 0 $ is a real number and $f_1, \ldots, f_k$ are continuous functions on $M$.   

\begin{claim}\label{cl:exercise}
There exists $ \delta > 0$ and continuous functions $f_1, \ldots, f_k : M \rightarrow \bbR$, satisfying $\Vert f_i \Vert_{\infty} \leq 1$ for each $i$, such that for every $\nu \in \Conv(\cE)$ we have $$ U_{\delta, f_1, \ldots, f_k}(\nu) \subset \cU.$$
\end{claim}  
\begin{proof}
The space of probability measures on $M$ is metrizable; more precisely, there exists a metric $d$ such that 
for every radius $r>0$, there exists $\delta>0$ and continuous functions $f_1, \ldots, f_k : M \rightarrow \bbR$ such that for every probability measure $\nu$, the set 
$U_{\delta, f_1, \ldots, f_k}(\nu)$ is contained in the ball of radius $r$ around $\nu$ (see~\cite{walters}, Theorem~6.4).
Now take $r$ equals  to the distance between $\Conv(\cE)$ and the complement of $\cU$. By compactness of $\Conv(\cE)$, $r$ is positive, and the claim follows.
\end{proof}

Recall that $W = \left( M \setminus \bigcup_{i} A_{i} \right) \subset B_{\varepsilon}(\Fix(S))$.   For each $x \in \Fix(S)$, write $W_x  = W \cap B_{\varepsilon}(x)$.  For small enough $\varepsilon$, the set $W_x$ is the connected component of $W$ which contains $x$.   Consider the probability measure  $\nu_{\varepsilon} \in \Conv(\cE)$ defined by
$$
\nu_{\varepsilon} := \sum_{x \in \Fix(S)} \mu(W_x) \delta_{x} + \alpha \mathrm{Vol},
$$
where $\delta_x$ denotes the Dirac measure at $x$.  Note that $\nu_{\varepsilon} \in\Conv(\cE) $ because $\sum_{x \in \Fix(S)} \mu(W_x) + \alpha =  \mu(W) + \mu(\cup A_i) = \mu(M) = 1$. 

\medskip

Proposition \ref{prop.main2} follows immediately from Claim \ref{cl:exercise} and  the next claim.

\begin{claim}\label{cl:mu_in_U}
Let $\delta, f_1, \ldots, f_k$ be as in Claim \ref{cl:exercise}.   If $\mu \in \cP(M)$ satisfies  Equation \eqref{eq:estimate_mu} for a sufficiently small value of $\varepsilon$, then $\mu \in U_{\delta, f_1, \ldots, f_k}(\nu_{\varepsilon})$.
\end{claim}
We will now provide a proof of the above claim.   Since the $A_i$'s are of diameter less than $\varepsilon > 0$,  the following two inequalities hold for sufficiently small values of $\varepsilon$:
$$ \left\vert \int f \, d\mu - \sum_i f(y_i) \mu(A_i) -  \sum_{x \in \Fix(S)} f(x) \mu(W_x) \right\vert < \frac{\delta}{4}, $$

$$ \left\vert \int f \, d\nu_{\varepsilon} - \sum_i f(y_i) \alpha \mathrm{Vol}(A_i) -  \sum_{x \in \Fix(S)} f(x) \mu(W_x) \right\vert < \frac{\delta}{4},$$
where  $f \in \{f_1, \ldots, f_k\}$ and $y_i \in A_i$.   It follows from the above two inequalities that  $\forall f \in \{f_1, \ldots, f_k\}$ we have
$$ \left\vert \int f \, d\nu_{\varepsilon} - \int f \, d\mu  \right\vert < \sum_i   \vert \alpha \mathrm{Vol}(A_i) - \mu(A_i) \vert + \frac{\delta}{2}.$$   Now, since $\left( M \setminus \bigcup_{i} A_{i} \right) \subset B_{\varepsilon}(\Fix(S))$, we have  $  \sum_i \left\vert \mathrm{Vol}(A_i) - \frac{1}{N} \right\vert < \varepsilon $, if $\varepsilon$ is taken to be sufficiently small. 
Thus, $$ \sum_i  \vert \alpha \mathrm{Vol}(A_i) - \mu(A_i) \vert \leq   \alpha \sum_i \vert \mathrm{Vol}(A_i) - \frac{1}{N} \vert +  \sum_i  \vert \frac{\alpha}{N}  - \mu(A_i) \vert    $$
$$ \leq  \varepsilon + \varepsilon.$$
Note that to obtain the last inequality we have used Inequality \eqref{eq:estimate_mu} and the fact that $\alpha <1$.  Finally, we conclude from the above that, if $\varepsilon$ is taken to be sufficiently small, then 
$$ \left\vert \int f \, d\nu_{\varepsilon} - \int f \, d\mu  \right\vert < \delta,$$
for every $f \in \{f_1, \ldots, f_k\}$.  This completes the proof of Claim \ref{cl:mu_in_U} and hence that of Proposition \ref{prop.main2}.
\end{proof}

\bigskip

%
%

 \bibliographystyle{alpha}
\bibliography{biblio}


{\small

\medskip
\noindent Fr\'ed\'eric Le Roux \\
\noindent Sorbonne Universit\'e, Universit\'e de Paris, CNRS, Institut de Math\'ematiques de Jussieu-Paris Rive Gauche, F-75005 Paris, France.\\
 {\it e-mail:} frederic.le-roux@imj-prg.fr
\medskip

\medskip
 \noindent Sobhan Seyfaddini\\
\noindent Sorbonne Universit\'e, Universit\'e de Paris, CNRS, Institut de Math\'ematiques de Jussieu-Paris Rive Gauche, F-75005 Paris, France.\\
 {\it e-mail:} sobhan.seyfaddini@imj-prg.fr

}

\end{document}